\documentclass[a4paper, 12pt]{article}
        \usepackage{amsmath, amssymb, amsfonts, latexsym, dsfont, enumerate}
        \usepackage{amsthm}
        \usepackage[latin1]{inputenc}
        \usepackage[all]{xy}
        \usepackage{algorithmic, algorithm}
        \usepackage[normalem]{ulem}

\DeclareMathOperator{\id}{id}
\DeclareMathOperator{\Mor}{Mor}
\DeclareMathOperator{\Sur}{Sur}
\DeclareMathOperator{\Hom}{Hom}

\DeclareMathOperator{\Image}{Im}
\DeclareMathOperator{\Ch}{Ch}
\DeclareMathOperator{\cross}{cr}
\DeclareMathOperator{\Sym}{Sym}

\DeclareMathOperator{\dev}{\intercal}
\DeclareMathOperator{\module}{\textrm{-}\bf{mod}}

\newtheorem{thm}{Theorem}[section]
\newtheorem{Lemma}[thm]{Lemma}
\newtheorem{cor}[thm]{Corollary}
\newtheorem{prop}[thm]{Proposition}

\theoremstyle{definition}
\newtheorem{defn}[thm]{Definition}
\newtheorem{note}[thm]{Notation}

\theoremstyle{remark}

\newtheorem{example}[thm]{Example}

\newtheorem{proc}[thm]{Procedure}

\algsetup{indent=2em}

\begin{document}

\title{An algorithmic approach to Dold-Puppe complexes}
\author{Ramesh Satkurunath and Bernhard K\"ock}
\date{\today}

\maketitle

\begin{quote}
{\footnotesize {\bf Abstract.} A Dold-Puppe complex is the image
$NF\Gamma(C.)$ of a chain complex $C.$ under the composition of
the functors $\Gamma$, $F$ and $N$ where $\Gamma$ and $N$ are
given by the Dold-Kan correspondence and $F$ is a not-necessarily
linear functor between two abelian categories. The first half of
this paper gives an algorithm that streamlines the calculation of
$\Gamma(C.)$. The second half gives an algorithm that allows the
explicit calculation of the Dold-Puppe complex $NF\Gamma(C.)$ in
terms of the cross-effect functors of $F$.

}

\end{quote}

\section*{Introduction}

Let \(R\) and \(S\) be rings.  The construction of the left
derived functors \(L_k F:R\module\rightarrow S\module\) of any
covariant right-exact functor \(F:R\module\rightarrow S\module\)
is achieved by applying three functors. The first functor
constructs a projective resolution \(P.\) of the \(R\)-module
\(M\) that we wish to calculate the derived functor of. Then the
functor \(F\) is applied to the resolution \(P.\) giving the chain
complex \(F(P.)\).  Lastly \(L_k F(M)\) is defined to be
\(H_k(F(P.))\), the \(k^{th}\) homology of the chain complex
\(F(P.)\). However for a given module \(M\) the projective
resolution of \(M\) is unique only up to chain-homotopy
equivalence, so this construction crucially depends on the fact
that \(F\) preserves chain-homotopies. In general this fact does
not hold when \(F\) is a nonlinear functor such as the \(l^{th}\)
symmetric power functor, \(\Sym^{l}\), or the \(l^{th}\) exterior
power functor, \(\Lambda^l\).  In the paper \cite{DP} Dold and
Puppe overcome this problem and define the derived functors of
non-linear functors by passing to the category of simplicial
complexes using the Dold-Kan correspondence.

The Dold-Kan correspondence gives a pair of functors \(\Gamma\) and \(N\)
that provide an equivalence between the category of bounded chain
complexes and the category of simplicial complexes;
under this correspondence chain homotopies correspond to simplicial homotopies.
Furthermore in the simplicial world all functors preserve simplicial homotopy (not just linear functors).
Because of this the above definition of the derived functors of \(F\) becomes well defined
for any functor when \(F(P.)\) is replaced by the complex \(NF\Gamma(P.)\).
We call chain complexes of the form \(NF\Gamma(C.)\) Dold-Puppe complexes,
for any bounded chain complex \(C.\).

Let \(R\) be a ring and let \(I\) be an ideal in \(R\) that is
locally generated by a non-zero divisor. If \(P.\) is a length-one
\(R\)-projective resolution of a projective \(R/I\)-module \(V\)
then the homology of the Dold-Puppe complex \(N\Sym^k\Gamma(P.)\),
\(k\ge 1\), has been explicitly computed in \cite{Ko1}. These
computations yield a very natural and new proof of the classical
Adams-Riemann-Roch theorem for regular closed immersions and hence
a new approach to the seminal Grothendieck-Riemann-Roch theorem
avoiding the comparatively involved deformation to the normal
cone, see \cite{Ko1}.

If \(C.\) is a chain complex of length bigger than~1 then the
calculation of the Dold-Puppe complex \(NF\Gamma(C.)\) is normally
too complicated to be performed on a couple of pieces of paper,
and the nature of the calculation means that errors easily creep
in. In this paper we analyse and elucidate its combinatorial
structure, and exploiting this structure that we have revealed we
develop an algorithm that computes this Dold-Puppe complex. We
hope that this explicit description of the Dold-Puppe complex will
help later work in calculating its homology, particularly in
concrete example situations. Moreover we expect that it will be
useful in computing maps between the homology of different
Dold-Puppe complexes, such as the plus and diagonal maps occurring
in \cite{Ko1}: for such calculations one often has to find
representatives on the complex level for elements of the homology.

We now describe the contents of each section in more detail.

In Section \ref{fdinDelta} we introduce an ordering on the set
\(\Mor([n],[k])\) of order-preserving maps between
\([n]:=\{0<1<\ldots<n\}\) and \([k]:=\{0<1<\ldots<k\}\) (see
Definition~\ref{order}). Basically the entire paper is based on
this crucial definition. We show at the end of Section~1 that
composition with the face maps \(\delta_i:[n-1]\rightarrow[n]\)
and degeneracy maps \(\sigma_i:[n]\rightarrow[n-1]\) is
``well-behaved" with respect to this ordering (see
Theorem~\ref{well behaved}).

The simplicial complex \(\Gamma(C.)\) is defined by
\[\Gamma(C.)_n=\bigoplus^{n}_{k=0}\bigoplus_{\mu\in\Sur([n],[k])}C_k\text{,}\]
so we have a copy of the direct summand \(C_k\) for each
surjective order-preserving map \(\mu:[n]\rightarrow[k]\). The
face and degeneracy operators in the simplicial complex
\(\Gamma(C.)\) are defined in terms of composition of \(\mu\) with
the maps \(\delta_i\) and \(\sigma_i\). In Section \ref{fdinGamma}
we  show how the results in Section \ref{fdinDelta} can be used to
streamline the calculation of the face and degeneracy operators in
the simplicial complex \(\Gamma(C.)\) (see
Theorem~\ref{bigGammathm} and Example~\ref{face and degeneracy
operator calculation example}).

In Section \ref{crosstheory} we summarise the results on
cross-effect functors that are needed for the final section.

The Dold-Puppe complex \(NF\Gamma(C.)\) is constructed by modding
out the images of the degeneracy operators in \(F\Gamma(C.)\). To
calculate this we  apply the theory of cross-effect functors to
decompose both the numerator and denominator into the direct sum
of cross-effect modules, the non-degenerate modules corresponding
to the terms that appear in the numerator but not in the
denominator. However the decomposition produces many, many terms
and seeing which are non-degenerate is far from obvious. In
Section \ref{honourabilitynstuff} we give a criterion that
identifies the non-degenerate terms (see Proposition~\ref{hon not
degen}).  Using the ordering we introduced in Section~1 we later
give an algorithm that constructs all relevant non-degenerate
terms, thus avoiding the need to check each of the many terms one
by one. We finally illustrate the methods developed in this paper
in the case when $C.$ is a chain complex of modules over a
commutative ring of length~2 and $F$ is the symmetric-square
functor (see Example~\ref{final}).

\section*{Notations}

Let \( \Delta \) be the category whose objects are the non-empty
finite totally ordered sets \( {[n] := \{0 < 1 < ... < n \}}\),
\(n \in \mathbb{N}\), and the set of morphisms, \(\Mor([n],[k])\),
between \([n]\) and \([k]\) consists of all the order-preserving
maps between them. Recall that for each \(i\in \{0,\dots,n\}\) the
face map \( \delta_i : [n-1] \rightarrow [n] \) is the unique
injective order-preserving map with \( \delta_i^{-1}(i) =
\emptyset,\) and for each \(i\in \{0,\dots,n-1\}\) the degeneracy
map \( \sigma_i: [n] \rightarrow [n-1]\) is the unique
order-preserving surjective map with \( \sigma_i^{-1}(i) = \{ i,
i+1 \} \). For a category~\( \mathcal{A} \), a simplicial object
\( A \) in \( \mathcal{A} \) is a contravariant functor \( A:
\Delta \rightarrow \mathcal{A} \).  We write \( A_n \) for \(
A([n]) \), \(d_i\) for the \emph{face operator} \( A(\delta_i)
:A_n \rightarrow A_{n-1} \), \( s_i \) for the \emph{degeneracy
operator} \( A(\sigma_i) :A_{n-1} \rightarrow A_n \) and
\(\Sur([n],[k])\) for the set of surjective morphisms between
\([n]\) and \([k]\).

\section{Partitions and composition with face/degeneracy maps in \(\Delta\)}\label{fdinDelta}

For the whole of this section let us fix the natural numbers \(n\)
and \(k\). In this section we introduce an ordering on
\(\Mor([n],[k])\), investigate the maps \(\mu\mapsto \mu\delta_i\)
and \(\nu\mapsto \nu\sigma_i\) between \(\Mor([n],[k])\) and
\(\Mor([n-1],[k])\) and show that these maps behave in a nice way
with respect to the introduced ordering.

This ordering will be used throughout this paper. In Section
\ref{fdinGamma} it will allow us to describe algorithms that
streamline the calculation of the face and degeneracy operators in
the simplicial complex \(\Gamma(C.)\) (for any bounded chain
complex $C.$). In Section 4 the ordering will help us to give an
algorithmic description of the Dold-Puppe
complex~\(NF\Gamma(C.)\).

\begin{defn}

For an \(n\)-tuple \( x:=(x_1, \ldots, x_n) \in \mathbb{N}^n \)
we write \(|x|\) for the sum \(\sum_{l=1}^n x_l,\) and
we call \(x\) \emph{a partition of \(m\) of length \(n\)} if \(|x|= m\).
If each \(x_i \not= 0 \) we call \(x\) a \emph{proper partition}, otherwise we call
\(x\) an \emph{improper partition}.  We write \(x_i\) for the \(i^{\textrm{th}}\) entry of \(x\).
\end{defn}

A function \(\mu:[n]\rightarrow[k]\) is determined by \(\mu^{-1}(0)\), \(\mu^{-1}(1),\;
\ldots,\; \mu^{-1}(k)\).  If \(\mu\) is a monotonically increasing function then the sets \(\mu^{-1}(0)\),
\(\mu^{-1}(1),\;\ldots,\; \mu^{-1}(k)\) consist of consecutive elements of \([n]\). Because of
this it is sufficient to know the sizes of these sets.
Hence we can think of a morphism \(\mu:[n]\rightarrow[k]\)
as a partition of \(n+1\) of length \(k+1\).
A surjective morphism would correspond to a
proper partition and a non-surjective morphism would correspond to
an improper partition.

\begin{note}\label{*def}
Let \(\mu\in\Mor([n],[k]).\) We write \(\mu^*\) for the partition
\((|\mu^{-1}(0)|,\ldots,|\mu^{-1}(n)|)\). Note that \(\mu^*_i =
|\mu^{-1}(i-1)|\).
\end{note}

\begin{Lemma}\label{sizesur} The cardinality of the set of
surjective order-preserving morphisms between between the sets
$[n]$ and $[k]$ is given by the binomial coefficient
$\binom{n}{k}:$
\[
|\Sur([n],[k])|= \binom{n}{k}\text{.}
\]
\end{Lemma}

\begin{proof}
If \(\mu:[n]\rightarrow[k]\) is a surjective morphism then the sets
\(\mu^{-1}(0)\), \(\mu^{-1}(1),\; \ldots,\; \mu^{-1}(k)\) are non-empty, disjoint, their union is \([n]\)
and they consist of consecutive elements of \([n]\).  So if we know the smallest elements
of \(\mu^{-1}(1)\), \(\mu^{-1}(2),\; \ldots,\; \mu^{-1}(k)\) then we have determined \(\mu\).
Since we know \(0 = \mu(0)\) the smallest elements are in the set \(\{1,...,n\}\).
So there are as many elements of \(\Sur([n],[k])\) as there are ways of
choosing \(k\) elements from a set of size \(n\).
\end{proof}

\begin{note}\label{inducedfns}
For \(i\in \{0,\ldots,n\}\) define \(\overline{\delta}_i :
\Mor([n],[k])\rightarrow \Mor([n-1],[k])\) by \(\mu \mapsto
\mu\delta_i\), and for \(i\in \{0,\ldots,n-1\}\) define
\(\overline{\sigma}_i:\Mor([n-1],[k])\rightarrow \Mor([n],[k])\)
by \(\nu \mapsto \nu\sigma_i\).  By abuse of notation we write
\(\Image \overline{\sigma}_i\) for
\(\overline{\sigma}_i(\Sur([n-1],[k]))\).
\end{note}


\begin{Lemma}\label{deltagamma=identity}
For all \(i \in \{0,\ldots,n-1\}\) we have \(\overline\delta_i\overline\sigma_i = \id\), and hence
\(\overline\sigma_i \) is injective and \(\overline\delta_i\) is surjective;
also \(\overline\delta_n\) is surjective.
\end{Lemma}

\begin{proof} The result follows directly from
\(\sigma_i\delta_i = \id\) for \(i \in \{0,\ldots,n-1\}\) and
from \(\sigma_{n-1}\delta_n = \id\text{.}\) \end{proof}

\begin{defn}\label{initial}
Let \(a\) be a partition of length \(k\)
and \(x\) a partition of length \(l\le k\).
Then we call \(x\) an \emph{initial partition of $a$} if
\(x_i = a_i\) for \( 1\le i\le l.\)  We write \(a=(x,y)\)
where \(y\) is the partition of length \(k-l\) defined by \(y_i = a_{i+l}\) for \(1 \le i \le k-l\).
(Note we may allow either \(x\) or \(y\) to be the empty partition.)
\end{defn}

Since knowing the effects of \(\overline{\delta}_i\) and
\(\overline{\sigma}_i \) are essential in calculating \(d_i\) and
\(s_i\) it is useful to have a quick way of working out the
partitions \((\mu\delta_i)^{*}\) and \((\mu\sigma_i)^{*}\) from
the partition~\(\mu^*\).

\begin{Lemma}\label{bareffect}
\begin{enumerate}[(a)]
\item Let \(\mu\in\Mor([n-1],[k])\) and \(i\in\{0,\ldots,n-1\}.\)
We write \(\mu^*=(x,d,y)\) with partitions \(x,y\) and a positive
integer \(d\) such that \(|x|<i+1\le |x|+d.\) Then the partition
\((\mu\sigma_i)^*\) is equal to \((x,d+1,y).\)


\item
Let \(\mu\in\Mor([n],[k])\) and \(i\in\{0,\ldots,n\}.\)
As above we write \(\mu^*=(x,d,y)\) so that
\(|x|<i+1\le |x|+d.\)
Then the partition \((\mu\delta_i)^*\) is equal to \((x,d-1,y).\)
\end{enumerate}
\end{Lemma}

\begin{proof}
It is clear that we can write $\mu^*$ in the stated way. Note that
$d\not= 0$, so $d-1$ is non-negative.

By definition for every \(\mu\) in \(\Mor([n],[k])\) we have
\(\mu_l^* = |\mu^{-1}(l-1)|\) and we also have
\((\mu\sigma_i)^{-1}(l-1) = \sigma_i^{-1} \mu^{-1}(l-1).\)
Recalling \(\sigma_i\) is the unique surjective map
\([n-1]\rightarrow[n]\) with \(\sigma_i^{-1}(i) = \{i, i+1\} \) we
see \(|(\mu\sigma_i)^{-1}(l-1)| = |\mu^{-1}(l-1)|\) if and only if
\(i \notin \mu^{-1}(l-1),\) and \(|(\mu\sigma_i)^{-1}(l-1)| =
|\mu^{-1}(l-1)| +1\) if and only if \(i \in \mu^{-1}(l-1);\) i.e.\
\(\mu^*_l=(\mu\sigma_i)^*_l\) if and only \(i \notin
\mu^{-1}(l-1)\), and \((\mu\sigma_i)^*_l = \mu^*_l +1\) if and
only if \(i \in \mu^{-1}(l-1).\)

Let $L$ be the length of $x.$
Remembering that \(i\) is the \((i+1)^{\textrm{th}}\) element of \([n]\)
we find that \(i \in \mu^{-1}(L)\) and so, by the last sentence of the previous paragraph,
we find \((\mu\sigma_i)^*=(x,d+1,y).\)

We similarly get our result for \(\delta_i.\)
\end{proof}

%
%

%

\begin{Lemma}\label{nonsur}
Let \(\mu\in\Sur([n],[k]),\) and let \(i\in\{0,\ldots,n\}.\) Then
the morphism \(\overline{\delta}_i(\mu)=\mu\delta_i\) is not
surjective if and only if the partition $\mu^*$ is of the form
\((x,1,y),\) where \(x\) is a partition of \(i\). In this case we
have the commutative diagram
\[
\xymatrix{
                                            &[n]\ar[dr]^{\mu}       &\\
[n-1]\ar[ur]^{\delta_i}\ar[dr]^{\hat\mu}    &                       &[k]\\
                                            &[k-1]\ar[ur]^{\delta_j}&\\
}
\]
where $\hat\mu$ is the surjection with \(\hat\mu^*=(x,y)\) and $j$
is the length of $x$; in particular $i=0$ if and only if $j=0$.
\end{Lemma}

\begin{proof}
The equivalence follows directly from Lemma~\ref{bareffect}(b).
The additional statements are easy to check.
\end{proof}

If \(a\) and \(b\) are both partitions of the same number
over the same number of places and \(x\) is an initial partition of both
them we call \(x\) a \emph{common initial partition of a and b}.
Because \(a\) and \(b\) are of finite length there must be some longest common
initial partition (even if it is of length 0, or it is equal to \(a\)).

\begin{defn}\label{order}
If \(x\) is the longest common initial partition of \(a=(x,y)\)
and \(b=(x,z)\) then we say \(a<b\) if and only if \(y_1<z_1\).
This gives the lexicographic ordering on the set of partitions and
finally, via the bijection $\mu \mapsto \mu^*$, a total  order on
\(\Mor([n],[k])\).
\end{defn}

%
%
%
%

\begin{note}
For \(i\in \{0\ldots,n\}\) let
$$S^{n,k}_i := \{ \mu \in
\Sur([n],[k]) \,|\, \mu^* \textrm{ is of the form } (x,y) \textrm{
where } |x|=i+1\}$$ and let
$$\widetilde{S^{n,k}_i}
:=\{\mu\in\Sur([n],[k])\,|\,\mu^* \textrm{ is of the form }
(x,1,y) \textrm{ where } |x|=i\}.$$
\end{note}
Note that \(\widetilde{S^{n,k}_i}\subset S^{n,k}_i\) and
Lemma~\ref{nonsur} tells us that the set \(\widetilde{S^{n,k}_i}\)
coincides with the set
\(\{\mu\in\Sur([n],[k])\,|\,\overline{\delta}_i(\mu) \text{ is not
a surjection}\}.\)

\begin{Lemma}\label{sizesin}
For each \(i\in\{0,\ldots,n-1\}\) we have \(|S^{n,k}_i| =
\binom{n-1}{k-1}\). Furthermore for each \(i\in\{1,\ldots,n-1\}\)
we have \(|\widetilde{S^{n,k}_i}|=\binom{n-2}{k-2}\) and finally
\(|\widetilde{S^{n,k}_n}|=\binom{n-1}{k-1}\).
\end{Lemma}

Note in the statement above, if the lower entry of a binomial
coefficient is negative, then the binomial coefficient is meant to
be $0$.

\begin{proof}
If \(\mu \in S^{n,k}_i\) then for some \(l\) we have \(i\) is the
maximal element of \(\mu^{-1}(l)\), furthermore we know that \(n\)
is the maximal element of \(\mu^{-1}(k)\).  Therefore choosing an
element \(\mu\) of \(S^{n,k}_i\) amounts to the same as choosing
the maximal elements for all but one of the sets
\(\mu^{-1}(0),\ldots,\mu^{-1}(k-1)\) from the \(n-1\) elements of
\([n]\setminus \{i,n\}\). Hence \(|S^{n,k}_i| =
\binom{n-1}{k-1}\).

For \(i\in\{1,\ldots,n-1\}\) if \(\mu \in \widetilde{S^{n,k}_i}\)
then for some \(l\) we have \(i-1\) is the maximal element of
\(\mu^{-1}(l),\) and also \(i\) is the maximal element of
\(\mu^{-1}(l+1)\), i.e.\ choosing an element \(\mu\) of
\(\widetilde{S^{n,k}_i}\) amounts to the same as choosing the
maximal elements for all but two of the sets
\(\mu^{-1}(0),\ldots,\mu^{-1}(k-1)\) from the \(n-2\) elements of
\([n]\setminus\{i-1,i,n\}\). Hence \(|\widetilde{S^{n,k}_i}| =
\binom{n-2}{k-2}\).

For the last statement we merely observe that
\(\widetilde{S^{n,k}_n}=S^{n,k}_{n-1}\) and use the first result.
\end{proof}

\begin{prop}\label{sur=Imgammaunionsin}
For each \(i\in\{0\ldots,n-1\}\) the set \(\Sur([n],[k])\) is the
disjoint union of \(S^{n,k}_i\) and \(\Image\overline{\sigma}_i:\)
\begin{displaymath}
\Sur([n],[k]) = S^{n,k}_i \amalg \Image\overline{\sigma}_i.
\end{displaymath}
\end{prop}

Note \(S_n^{n,k} = \Sur([n],[k])\) and there is no map
\(\overline{\sigma}_n.\)

\begin{proof}
First we prove \(S^{n,k}_i\) and \(\Image\overline{\sigma}_i\) are
disjoint. Let \(\mu \in \Sur([n],[k])\). The partition \(\mu^*\)
has an initial partition of \(i+1\) if and only if there is some
\(l\) such that \(i\) is the maximal element of \(\mu^{-1}(l)\)
(remember \(i\) is the \((i+1)^\textrm{th}\) element of \([n]\)).
If \(i\) is the maximal element of \(\mu^{-1}(l)\) then \(\mu(i)
\ne \mu(i+1)\). But \(\mu \in \Image\overline{\sigma}_i\) means
that for some \(\nu\in \Sur([n-1],[k])\) we have \(\mu =
\nu\sigma_i\). So
\(\mu(i)=\nu\sigma_i(i)=\nu(i)=\nu\sigma_i(i+1)=\mu(i+1)\).
Therefore \(\mu\) cannot be both in \(S^{n,k}_i\) and
$\Image\overline{\sigma}_i$.

Now we prove that the union of \(S^{n,k}_i\) and
\(\Image\overline{\sigma}_i\) form the whole of \(\Sur([n],[k])\)
by using a counting argument. We know that \(S^{n,k}_i \cap
\Image\overline{\sigma}_i = \emptyset\) so \(|S^{n,k}_i \cup
\Image\overline{\sigma}_i| = |S^{n,k}_i| +
|\Image\overline{\sigma}_i|\). Lemma~\ref{deltagamma=identity}
tells us that \(\overline{\sigma}_i\) is injective.  From this we
see that \(|S^{n,k}_i| + |\Image\overline{\sigma}_i| = |S^{n,k}_i|
+ |\Sur([n-1],[k])|\) and using Lemmas \ref{sizesur} and
\ref{sizesin} we obtain \( |S^{n,k}_i| + |\Sur([n-1],[k])| =
\binom{n-1}{k-1} + \binom{n-1}{k} = \binom{n}{k} =
|\Sur([n],[k])|\), as desired.
\end{proof}

\begin{thm} \label{well behaved}
\begin{enumerate}[(a)]
\item For each \(i \in \{0,\ldots,n-1\} \) the map
\(\overline{\sigma}_i:\Mor([n-1],[k])\rightarrow\Mor([n],[k])\) is
strictly order-preserving. \item For each \(i\in\{0,\ldots,n-1\}\)
the map \(\overline{\delta}_i\) is strictly order-preserving on
both \(\Image\overline{\sigma}_i\) and \(S^{n,k}_i\), and
\(\overline{\delta}_n\) is strictly order-preserving on
\(\Sur([n],[k])=S^{n,k}_n\).
\end{enumerate}
\end{thm}

Note that while \(\overline{\delta}_i\) is order-preserving on
these two complementary sets of \(\Sur([n],[k])\) it is \emph{not}
order-preserving on the whole of \(\Sur([n],[k])\); for an
illustration of this look at the calculation at the end of Section
\ref{fdinGamma}.

\begin{proof}

(a) Suppose \(\mu,\nu\in \Mor([n-1],[k])\) and \(\mu<\nu\). As in
Lemma~\ref{bareffect} we write the partition \(\mu^*\) in the form
\((x,d,y)\) where \(|x|<i+1\le |x|+d.\) Let \(a\) be the longest
common partition of \(\mu^*\) and \(\nu^*,\) so \(\mu^*=(a,b)\)
and \(\nu^*=(a,c)\) for appropriate partitions \(b\) and \(c\)
with \(b_1<c_1.\) We will show the desired inequality
\((\mu\sigma_i)^*<(\nu\sigma_i)^*\) by distinguishing three cases:
(i) $a$ is longer than $x,$ (ii) $a$ has the same length as $x$
and (iii) $a$ is shorter than~$x.$

(i) If $a$ is longer than $x$ then we can write $a$ in the form
$(x,d,w)$ for some (possibly empty) partition $w.$  Then
\(\mu^*=(x,d,w,b)\) and \(\nu^*=(x,d,w,c)\) and hence by
Lemma~\ref{bareffect} we see that \((\mu\sigma_i)^*=(x,d+1,w,b)\)
and \((\nu\sigma_i)^*=(x,d+1,w,c)\). So the longest common initial
partition of \((\mu\sigma_i)^*\) and \((\nu\sigma_i)^*\) is
\((x,d+1,w)\) and, since \(b_1<c_1,\) we see that
\((\mu\sigma_i)^*<(\nu\sigma_i)^*.\)

(ii) If \(a\) has the same length as \(x\) (i.e.\ if \(a=x\)) then
\(d=b_1\) and since we have \(b_1<c_1\) we see \(i+1\le
|x|+d<|x|+c_1.\) Using Lemma~\ref{bareffect} we obtain
\((\mu\sigma_i)^*=(x,b_1+1,y)\) and
\((\nu\sigma_i)^*=(x,c_1+1,z)\) for appropriate partitions $y$ and
$z.$ So the longest common initial partition of
\((\mu\sigma_i)^*\) and \((\nu\sigma_i)^*\) is \(x\) and, since
\(b_1+1<c_1+1,\) we see that \((\mu\sigma_i)^*<(\nu\sigma_i)^*.\)

(iii) If $x$ is longer than $a$ we write \(x=(a,x')\) for some
non-empty partition $x'$. Then \(\mu^*=(x,d,y)=(a,x',d,y).\) As in
Lemma~\ref{bareffect} we write \(\nu^*=(w,d',z)\) where $|w| < i+1
\le |w|+d'$. We know that \(|a|\le |x|<i+1\) and \(a\) is an
initial partition of \(\nu^*\) so \(w=(a,w')\) for some possibly
empty partition $w'.$ We now show the desired inequality
\((\mu\sigma_i)^*<(\nu\sigma_i)^*\) by distinguishing two
subcases: ($\alpha$) $w'$ is non-empty, ($\beta$) $w'$ is empty.

($\alpha$) If $w'$ is not empty then \(\mu^*=(a,x',d,y)\) and
\(\nu^*=(a,w',d',z).\) Since \(\mu^*<\nu^*\) we find that
\(x'_1<w'_1.\) Applying Lemma~\ref{bareffect} we find that
\((\mu\sigma_i)^*=(a,x',d+1,y)\) and
\((\nu\sigma_i)^*=(a,w',d'+1,z).\) So the longest common initial
partition of \((\mu\sigma_i)^*\) and \((\nu\sigma_i)^*\) is \(a\)
and, since \(x'_1<w'_1\) we see that
\((\mu\sigma_i)^*<(\nu\sigma_i)^*.\)

($\beta$) If $w'$ is empty then \(\mu^*=(a,x',d,y)\) and
\(\nu^*=(a,d',z)\) where \(|a|<i+1\le|a|+d'.\) Since
\(\mu^*<\nu^*\) we see that \(x'_1<d'.\) Applying
Lemma~\ref{bareffect} we find that
\((\mu\sigma_i)^*=(a,x',d+1,y)\) and
\((\nu\sigma_i)^*=(a,d'+1,z).\) So the longest common initial
partition of \((\mu\sigma_i)^*\) and \((\nu\sigma_i)^*\) is \(a\)
and, since \(x_1<d'<d'+1\) we see that
\((\mu\sigma_i)^*<(\nu\sigma_i)^*.\)

(b) That \(\overline{\delta}_i\) is order-preserving on
\(\Image\overline{\sigma}_i\) follows directly from
Lemma~\ref{deltagamma=identity} and part~(a). Although (the first
half of) the proof that $\overline{\delta}_i$ is strictly
order-preserving on $S^{n,k}_i$ is pretty similar to (the first
half of) the proof of part~(a) we include all details for the
reader's convenience.

Suppose \(\mu, \nu \in S^{n,k}_i\) with \(\mu<\nu\). As in
Lemma~\ref{bareffect} we write the partition \(\mu^*\) in the form
\((x,d,y)\) where \(|x|<i+1\le |x|+d.\) Let \(a\) be the longest
common partition of \(\mu^*\) and \(\nu^*,\) so \(\mu^*=(a,b)\)
and \(\nu^*=(a,c)\) for appropriate partitions \(b\) and \(c\)
with \(b_1<c_1.\) We will now show the desired inequality
\((\mu\delta_i)^*<(\nu\delta_i)^*\) by distinguishing three cases:
(i)~$a$ is longer than $x,$ (ii)~$a$ has the same length as $x$
and (iii)~$a$ is shorter than $x$. Only case~(iii) will make use
of the assumption that $\mu, \nu \in S^{n,k}_i$.

(i) If $a$ is longer than $x$ then we can write $a$ in the form
$(x,d,w)$ for some (possibly empty) partition $w.$  Then
\(\mu^*=(x,d,w,b)\) and \(\nu^*=(x,d,w,c)\) and hence by
Lemma~\ref{bareffect} we see that \((\mu\delta_i)^*=(x,d-1,w,b)\)
and \((\nu\delta_i)^*=(x,d-1,w,c)\). So the longest common initial
partition of \((\mu\delta_i)^*\) and \((\nu\delta_i)^*\) is
\((x,d-1,w)\) and, since \(b_1<c_1,\) we see that
\((\mu\delta_i)^*<(\nu\delta_i)^*.\)

(ii) If \(a\) has the same length as \(x\) (i.e.\ if \(a=x\)) then
\(d=b_1\) and since we have \(b_1<c_1\) we see \(i+1\le
|x|+d<|x|+c_1.\) Using Lemma~\ref{bareffect} we obtain
\((\mu\delta_i)^*=(x,b_1-1,y)\) and
\((\nu\delta_i)^*=(x,c_1-1,z)\) for appropriate partitions $y$ and
$z.$ So the longest common initial partition of
\((\mu\delta_i)^*\) and \((\nu\delta_i)^*\) is \(x\) and, since
\(b_1-1<c_1-1,\) we see that \((\mu\delta_i)^*<(\nu\delta_i)^*.\)

(iii) If $x$ is longer than $a$ we write \(x=(a,x')\) for some
non-empty partition $x'$. Then \(\mu^*=(x,d,y)=(a,x',d,y).\) As in
Lemma~\ref{bareffect} we write \(\nu^*=(w,d',z)\) with $|w| < i+1
\le |w|+d'$. We know that \(|a|\le |x|<i+1\) and \(a\) is an
initial partition of \(\nu^*\) so \(w=(a,w')\) for some possibly
empty partition $w'.$ We now show the desired inequality
\((\mu\sigma_i)^*<(\nu\sigma_i)^*\) by distinguishing two
subcases: ($\alpha$) $w'$ is non-empty, ($\beta$) $w'$ is empty.

($\alpha$) If $w'$ is not empty then \(\mu^*=(a,x',d,y)\) and
\(\nu^*=(a,w',d',z).\) Since \(\mu^*<\nu^*\) we find that
\(x'_1<w'_1.\) Applying Lemma~\ref{bareffect} we find that
\((\mu\delta_i)^*=(a,x',d-1,y)\) and
\((\nu\delta_i)^*=(a,w',d'-1,z).\) So the longest common initial
partition of \((\mu\delta_i)^*\) and \((\nu\delta_i)^*\) is \(a\)
and, since \(x'_1<w'_1,\) we see that
\((\mu\delta_i)^*<(\nu\delta_i)^*.\)

($\beta$) If $w'$ is empty then \(\mu^*=(a,x',d,y)\) and
\(\nu^*=(a,d',z)\) where \(|a|<i+1\le|a|+d'.\) Since
\(\mu^*<\nu^*\) we see that \(x'_1<d'.\) Applying
Lemma~\ref{bareffect} we find that
\((\mu\delta_i)^*=(a,x',d-1,y)\) and
\((\nu\delta_i)^*=(a,d'-1,z).\) As \(x'_1<d'\) we have either
\(x'_1<d'-1\) or \(x'_1=d'-1\).

If \(x'_1<d'-1\) then the longest common initial partition of \((\mu\delta_i)^*\) and \((\nu\delta_i)^*\) is \(a\) and,
since \(x'_1<d'-1,\) we have \((\mu\delta_i)^*<(\nu\delta_i)^*.\)

If \(x'_1=d'-1,\) then we observe: we have written \(\nu^*\) as
\((a,d',z)\) so that \(|a|<i+1\le |a|+d',\) but \(\nu\in
S^{n,k}_i\) so \(\nu^*\) begins with a partition of \(i+1,\) hence
\(|a|+d'=i+1.\)  Now \(i+1=|a|+d'=|a|+x'_1+1,\) so \(|a|+x'_1=i,\)
i.e.\ the partition \((a,x'_1)\) (which is an initial partition of
\(\mu^*\)) is a partition of \(i.\) But \(\mu\in S^{n,k}_i\) so
\(\mu\) begins with a partition of \(i+1\) and \(\mu^*\) is a
proper partition, so \(\mu^*\) begins with the partition
\((a,x'_1,1)\), i.e.\ \(x=(a,x'_1)\) and \(d=1.\) So
\(\mu^*=(a,x'_1,1,y)\)  and \(\nu^*=(a,d',z).\)

%

By Lemma~\ref{bareffect} we find that
\((\mu\delta_i)^*=(a,x'_1,0,y)\) and
\((\nu\delta_i)^*=(a,d'-1,z)=(a,x'_1,z).\) Since all the entries
of \(\nu^*\) are positive we have \(z_1>0.\) So the longest common
initial partition of \((\mu\delta_i)^*\) and \((\nu\delta_i)^*\)
is \((a,x'_1)\) and, since \(0<z_1,\) we find that
\((\mu\delta_i)^*<(\nu\delta_i)^*.\)
\end{proof}

\section{The face and degeneracy operators in the simplicial object \(\Gamma(C.)\)}\label{fdinGamma}

For an abelian category \(\mathcal{A}\) the Dold-Kan
correspondence gives two mutually inverse functors \(\Gamma\) and
\(N\) between the category  \(\Ch_{\ge0}(\mathcal{A})\) of bounded
chain complexes and the category \(\mathcal{SA}\) of simplicial
objects in \(\mathcal{A}\). For a chain complex \(C.\in\Ch_{\ge
0}(\mathcal{A})\) the functor \(\Gamma(C.)\) is usually defined by
\(\Gamma(C.)_n=\bigoplus^{n}_{k=0}\bigoplus_{\sigma\in\Sur([n],[k])}C_k\).
So \(\Gamma (C.)\) contains \(|\Sur([n],[k])|\) copies of \(C_k\)
and these copies are indexed by elements of \(\Sur([n],[k])\). We
write \(\Gamma(C.)_{n,k}\) to denote
\(\bigoplus_{\sigma\in\Sur([n],[k])}C_k\) considered as a sub-sum
of \(\Gamma(C.)_n.\)

The effect of the degeneracy operator
\(s_i:\Gamma(C.)_{n-1}\rightarrow\Gamma(C.)_n\) on the copy of
\(C_k\) indexed by \(\mu\in\Sur([n-1],[k])\) is to identify it
with the copy of \(C_k\in\Gamma(C.)_n\) indexed by
\(\overline\sigma_i(\mu)\) (cf.\ Notation \ref{inducedfns}).

The effect of the face operator
\(d_i:\Gamma(C.)_n\rightarrow\Gamma(C.)_{n-1}\) on the copy of
\(C_k\) indexed by \(\mu\in\Sur([n],[k])\) depends on the nature
of \(\overline\delta_i(\mu)\) (cf.\ Notation \ref{inducedfns}):
\begin{itemize}
\item If \(\overline\delta_i(\mu)\) is surjective then \(C_k\) is identified with the copy of \(C_k\)
indexed by \(\overline\delta_i(\mu)\);

\item If \(\overline\delta_i(\mu)\) is not surjective, and
\(\overline\delta_i(\mu) = \delta_0\hat\mu\) for some \(\hat\mu
\in \Sur([n-1],[k-1]) \) (cf.\ Lemma~\ref{nonsur}) then \(d_i\)
maps the copy of \(C_k\) indexed by \(\mu\) to the copy of
\(C_{k-1}\) indexed by \(\hat\mu\) with the same action as the
differential of \(C.\);

\item If \(\overline\delta_i(\mu)\) is not surjective, and
\(\overline\delta_i(\mu)= \delta_j\hat\mu\) for some \(\hat\mu \in
\Sur([n-1],[k-1]) \) and for some \( j \ne 0 \) (cf.\
Lemma~\ref{nonsur}) then \(C_k\) is mapped to 0.
\end{itemize}

This can be expressed more concisely in symbols than in words.
For \(\mu\in\Sur([n],[k])\) we write \(C_{k,\mu}\) to denote the copy of \(C_k\)
in \(\bigoplus_{\sigma\in\Sur([n],[k])}C_k\) that is contributed by \(\mu\) and
also, for \(m\in C_k\), we write \((m,\mu)\) to denote \(m\in C_{k,\mu}\).
The face and degeneracy maps in \( \Gamma(C.)\) are defined as follows:

\vskip 0.25cm

\(s_i(m,\mu):=(m,\overline{\sigma}_i(\mu)),\)

\vskip 0.25cm

\(d_i(m,\mu):=
\begin{cases}
(m,\overline{\delta}_i(\mu)) & \text{if \(\overline{\delta}_i(\mu)\) is surjective}\\
(\partial(m),\hat\mu)        & \text{if \(\overline{\delta}_i(\mu)=\delta_0\hat\mu\) with \(\hat\mu \in \Sur([n-1],[k-1])\)}\\
0                            & \text{if \(\overline{\delta}_i(\mu)=\delta_j\hat\mu\) with \(\hat\mu \in \Sur([n-1],[k-1])\) and }j\ne0.\\
\end{cases}
\)

\vskip 0.25cm

The object of this section is to rewrite these expressions using
results from the previous section and to thereby make the
calculation of the face and degeneracy operators simpler.

Lemma~\ref{sizesur} tells us that for natural numbers \(n\) and
\(k\)
\[\Gamma(C.)_n = \Gamma(C.)_{n,0} \oplus \Gamma(C.)_{n,1} \oplus \ldots \oplus \Gamma(C.)_{n,n} =  C_0^{\binom{n}{0}} \oplus C_1^{\binom{n}{1}} \oplus \ldots \oplus C_n^{\binom{n}{n}};\]
again each copy of \(C_k\) is
indexed by the element of \(\Sur([n],[k])\) that contributes it.
But now we can use the ordering on \(\Sur([n],[k])\) that we defined in Section \ref{fdinDelta}
to order the copies of \(C_k\).
Because of this we will tend to use the ordinal associated to
\(\mu\in\Sur([n],[k])\) instead of \(\mu\) to index a copy of \(C_k,\)
i.e.\ if \(\mu\) is the \(m^\textrm{th}\) element of \(\Sur([n],[k])\)
we will usually write \(C_{k,m}\) instead of \(C_{k,\mu}\).

Combining various results from the previous section we get the
following proposition. For \(n,k\in\mathbb{N}_0\) and
\(A\subset\Sur([n],[k])\) we write \(A^C\) for the complement of
\(A\) in the set \(\Sur([n],[k])\).

\begin{prop}\label{collatingstuffinDelta} Let $n >0$ and $k \in
\{0, \ldots, n\}$.

\begin{enumerate}[(a)]

\item 
\begin{enumerate}[(i)]
\item For each \(i\in \{0,\ldots,n-1\}\) the sets
\(\Sur([n-1],[k])\) and \((S^{n,k}_i)^C\) have the same
cardinality.

\item The sets \(S^{n,k}_0\) and \(\Sur([n-1],[k-1])\) have the
same cardinality.

\item For each \(i\in \{1,\ldots,n\}\) the sets
\(S^{n-1,k}_{i-1}\) and
\(S^{n,k}_i\setminus\widetilde{S^{n,k}_i}\) have the same
cardinality.

\end{enumerate}

\item 
For each \(i\in \{0,\ldots,n-1\}\) the map
\(\overline{\sigma}_i:\Sur([n-1],[k])\rightarrow \Sur([n],[k])\)
sends the \(l^\textrm{th}\) element of \(\Sur([n-1],[k])\) to the
\(l^\textrm{th}\) element of \((S^{n,k}_i)^C\).

\item 
\begin{enumerate}[(i)]

\item If \(\mu\in S^{n,k}_0\) then for some \(\hat \mu \in
\Sur([n-1],[k-1])\) we have
\(\overline{\delta}_0(\mu)=\delta_0\hat \mu\). Moreover the map
\(\mu\mapsto\hat \mu\) acts on \(S^{n,k}_0\) by sending the
\(l^\textrm{th}\) element of \(S^{n,k}_0\) to the
\(l^\textrm{th}\) element of \({\Sur([n-1],[k-1])}\).

\item For each \(i\in\{0,\ldots,n-1\}\) the map
\(\overline{\delta}_i:\Sur([n],[k])\rightarrow \Mor([n-1],[k])\)
acts on the set \((S^{n,k}_i)^C\) by sending the \(l^\textrm{th}\)
element of \((S^{n,k}_i)^C\) to the \(l^\textrm{th}\) element of
\(\Sur([n-1],[k])\).

\item For each \(i\in\{1,\ldots,n\}\) the map
\(\overline{\delta}_i:\Sur([n],[k])\rightarrow \Mor([n-1],[k])\)
acts on the set \(S^{n,k}_i\setminus\widetilde{S^{n,k}_i}\) by
sending the \(l^\textrm{th}\) element of
\(S^{n,k}_i\setminus\widetilde{S^{n,k}_i}\) to the
\(l^\textrm{th}\) element of \(S^{n-1,k}_{i-1}.\)
\end{enumerate}
\end{enumerate}
\end{prop}

Part~(a) of this proposition ensures that the later statements are
well defined.

Note for \(i\ne 0\) we do not describe the action of
\(\overline\delta_i\) on \(\widetilde{S^{n,k}_i}\) because from
Lemma~\ref{nonsur} we know for \(\mu\in \widetilde{S^{n,k}_i}\)
the map~\(\overline\delta_i(\mu)\) will be a non-surjection equal
to \(\delta_j\hat\mu\) where \(j\ne 0\), hence the action of
\(d_i\) on \(C_{k,\mu}\) will just be the zero map (see the
definition of \(\Gamma\) at the beginning of this section).

\begin{proof}
Part~(a)(i) follows from Proposition~\ref{sur=Imgammaunionsin} and
the injectivity of \(\overline{\sigma}_i\)
(Lemma~\ref{deltagamma=identity}). Part~(a)(ii) follows from
Lemmas \ref{sizesur} and \ref{sizesin}. Lemma~\ref{sizesin}
furthermore tells us that for \(i\in \{1,\ldots,n-1\}\) we have
\(|S^{n,k}_i| = \binom{n-1}{k-1}\) and that
\(|\widetilde{S^{n,k}_i}| = \binom{n-2}{k-2}\), and therefore
\(|S^{n,k}_i\setminus\widetilde{S^{n,k}_i}| = \binom{n-1}{k-1} -
\binom{n-2}{k-2} = \binom{n-2}{k-1} =|S^{n-1,k}_{i-1}|\) (the
final step is given by Lemma~\ref{sizesin} again). Furthermore
\(S_n^{n,k} = \Sur([n],[k])\) and by using Lemma~\ref{sizesin}
twice we see that \(|\widetilde{S^{n,k}_n}|=\binom{n-1}{k-1}\) so
\(|S^{n,k}_n\setminus\widetilde{S^{n,k}_n}|=\binom{n}{k}-\binom{n-1}{k-1}=\binom{n-1}{k}=|S_{n-1}^{n-1,k}|.\)
So we have shown part~(a)(iii) of this theorem for all \(i\in
\{1,\ldots,n\}\).

Part~(b) is seen by applying Proposition~\ref{sur=Imgammaunionsin}
and Theorem~\ref{well behaved}(a) to part~(a)(i).


If \(\mu\in S^{n,k}_0\) then \(\mu^*\) is of the form \((1,y)\)
for an appropriate partition \(y.\) Applying Lemma~\ref{nonsur}
gives us the first sentence of part~(c)(i), and also tells us that
\(\hat\mu^*\) is the partition \(y\). Clearly the map that sends
\((1,y)\) to \(y\) is order-preserving. Now using (a)(ii) we get
(c)(i). By applying Theorem~\ref{well behaved}(b) to part~(a)(i)
we get part (c)(ii). Finally part~(c)(iii) follows by applying
Theorem~\ref{well behaved}(b) to part~(a)(iii) of this statement.
Note that $\overline{\delta}_i(S^{n,k}_i \setminus
\widetilde{S^{n,k}_i}) \subseteq S^{n-1,k}_{i-1}$.
\end{proof}

\begin{thm}\label{bigGammathm}
Let \(n>0\).

\begin{enumerate}[(a)] 

\item Let \(i\in\{0,\ldots,n-1\},\) fix \(k\in\{0,\ldots,n\}\) and
let \(\underline{c}\in\Gamma(C.)_{n-1,k}\), then we have
\(s_i(\underline{c})\in\Gamma(C.)_{n,k}.\) More precisely, write
\(\underline{c}=(c_1,\ldots,c_{\binom{n-1}{k}})\) and
\(s_i(\underline{c})=(b_1,\ldots,b_{\binom{n}{k}})\); then
\(s_i(\underline{c})\) is given by the following relations:
\begin{enumerate}[(i)]
\item If the \(l^\textrm{th}\) element of \(\Sur([n],[k])\) is an
element of \(S^{n,k}_{i}\) then then \(b_{l}=0\).

\item If the \(l^\textrm{th}\) element of \(\Sur([n],[k])\) is the
\(m^{\textrm{th}}\) element of  \((S^{n,k}_{i})^C\) then
\(b_{l}=c_{m}.\)
\end{enumerate}

\item Let
\(\underline{c}=(c_{k,l})_{k=0,\ldots,n;l=1,\ldots,\binom{n}{k}}
\in\Gamma(C.)_n.\) Then
\(d_0(\underline{c})=(b_{k,l})_{k=0,\ldots,n-1;l=1,\ldots,\binom{n-1}{k}}\in\Gamma(C.)_{n-1}\)
is given by the following relation:
\(b_{k,l}=\partial(c_{k+1,l})+c_{k,\binom{n-1}{k-1}+l}\).

\item Let \(i\in\{1,\ldots,n-1\},\) fix \(k\in\{0,\ldots,n\}\) and
let \(\underline{c}\in\Gamma(C.)_{n,k},\) then we have
\(d_i(\underline{c})\in\Gamma(C.)_{n-1,k}.\) More precisely, write
\(\underline{c}=(c_1,\ldots,c_{\binom{n}{k}})\) and
\(d_i(\underline{c})=(b_1,\ldots,b_{\binom{n-1}{k}})\); then
\(d_i(\underline{c})\) is given by the following relations:

\begin{enumerate}[(i)]
\item If the \(l^\textrm{th}\) element of \(\Sur([n-1],[k])\) is
an element of \((S^{n-1,k}_{i-1})^C\) then
\(b_{l}=c_{\alpha(l)}\), where \(\alpha(l)\) is the ordinal
associated with the \(l^{\textrm{th}}\) element of
\((S^{n,k}_{i})^C\).

\item If the \(l^\textrm{th}\) element of \(\Sur([n-1],[k])\) is
the \(m^{\textrm{th}}\) element of \(S^{n-1,k}_{i-1}\) then
\(b_{l}=c_{\alpha(l)}+c_{\beta(m)}\) where \(\alpha(l)\) is the
ordinal associated to the \(l^\textrm{th}\) element of
\((S^{n,k}_i)^C\) and \(\beta(m)\) is the ordinal associated with
the \(m^{\textrm{th}}\) element of
\(S^{n,k}_i\setminus\widetilde{S^{n,k}_i}.\)
\end{enumerate}

\item Fix \(k\in\{0,\ldots,n\}\) and let
\(\underline{c}\in\Gamma(C.)_{n,k},\) then we have
\(d_n(\underline{c})\in\Gamma(C.)_{n-1,k}.\) More precisely, write
\(\underline{c}=(c_1,\ldots,c_{\binom{n}{k}})\) and
\(d_n(\underline{c})=(b_1,\ldots,b_{\binom{n-1}{k}})\); then
\(d_n(\underline{c})\) is given by the following relation:

Let $\beta(l)$ denote the ordinal associated with the
$l^{\textrm{th}}$ element of $S^{n,k}_n \setminus
\widetilde{S^{n,k}_n} = \Sur([n],[k]) \setminus S_{n-1}^{n,k}$;
then $b_l = c_{\beta(l)}$.

\end{enumerate}

\end{thm}

\begin{proof}

Part~(a) follows from Proposition~\ref{collatingstuffinDelta}(b).
To prove part~(b) we first observe that
\(S^{n,k}_0=\{\mu\in\Sur([n],[k])\mid \mu^*=(1,x) \text{ where }
|x|=n)\}\); so \(S^{n,k}_0\) consists of the first
\(\binom{n-1}{k-1}\)~elements of \(\Sur([n],[k])\). Now part~(b)
follows from Proposition~\ref{collatingstuffinDelta}(c)(i) and
(c)(ii). Part~(c) follows from Lemma~\ref{nonsur} and
Proposition~\ref{collatingstuffinDelta}(c)(ii) and (c)(iii).
Finally part~(d) follows from Lemma~\ref{nonsur} and
Proposition~\ref{collatingstuffinDelta}(c)(iii).
\end{proof}

In Example~\ref{face and degeneracy operator calculation example}
below we look at the case when the chain complex $C.$ is of
length~$2$, to help elucidate the previous results. But first we
give some general instructions on how to read that example.

While part~(b) of the previous theorem is a very explicit formula
which allows to instantly describe the action of the face operator
$d_0$ we first need to calculate the sets $S^{n,k}_i$ (and
$\widetilde{S^{n,k}_i}$) to be able to use the other parts for
describing the degeneracy operators and the other face operators.

For each \(n\) that we are concerned with (the position in the
simplicial complex $\Gamma(C.)$) and each
\(k\in\{1,\ldots,\min(n,l)\}\) (where \(l\) stands for the length
of the chain complex $C.$) we draw a table to help us determine
these sets. We label the columns of the table by the possible
values of \(i\) (0 through to \(n\)). We label the rows of the
table with both the partition and the ordinal associated with the
elements of \(\Sur([n],[k])\). If a cell in the table has its
column labelled by \(i\) and its row is labelled by a partition
\(\mu^*\) that has an initial partition of \(i+1\) then we mark
the cell with a \(\times\) mark, if that initial partition ends
with a 1 then we also mark the cell with a \(^*\). So if a cell is
marked with a \(\times\) mark then the corresponding surjection
\(\mu\) is an element of the set \(S^{n,k}_i\), if the cell is
also marked with a \(^*\) then \(\mu\) is an element of the set
\(\widetilde{S^{n,k}_i}\). We do not draw any tables for $k=0$
because all face and degeneracy operators act just as the identity
on the single copy of $C_0$ in $\Gamma(C.)_n$.

We now explain how to use the tables we have made to calculate the
degeneracy operators. For this paragraph we fix
\(i\in\{0,\ldots,n-1\}\) and \(k\in\{0,\ldots,n\},\) let
\(\underline{c}\in\Gamma(C.)_{n-1,k}\) and write
\(\underline{c}=(c_1,\ldots,c_{\binom{n-1}{k}}).\) The vector
\(s_i(\underline{c})\in\Gamma(C.)_{n,k}\) is an
\(\binom{n}{k}\)-tuple. By Theorem~\ref{bigGammathm}(a) the
entries of \(s_i(\underline{c})\) are either \(0\) or one of
\(c_1,\ldots,c_{\binom{n-1}{k}}\); more specifically
\(c_1,\ldots,c_{\binom{n-1}{k}}\) each occur once in
\(s_i(\underline{c})\) and occur {\em in order}, with zeroes in
all the other entries.  We find where the zeroes are in
\(s_i(\underline{c})\) by looking at the column labelled \(i\) in
the table we made for \((n,k)\); if there is an \(\times\) in the
\(l^{\rm{th}}\) row of this column, then (by
Theorem~\ref{bigGammathm}(a)(i)) the \(l^{\rm{th}}\) entry of
\(s_i(\underline{c})\) is zero.

We now explain how to calculate the face operator \(d_n.\) For
this paragraph we fix \(k\in\{0,\ldots,n\},\) let
\(\underline{c}\in\Gamma(C.)_{n,k}\) and write
\(\underline{c}=(c_1,\ldots,c_{\binom{n}{k}}).\) If \(k=n\) then
\(\Gamma(C.)_{n-1,k}\) is just the zero module, so
\(d_n(\underline{c})=0\). In general, the vector
\(d_n(\underline{c})\in\Gamma(C.)_{n-1,k}\) is an
\(\binom{n-1}{k}\)-tuple. By Theorem~\ref{bigGammathm}(d) each
entry of \(d_i(\underline{c})\) is one of
\(c_1,\ldots,c_{\binom{n}{k}}\); more specifically
\(\binom{n-1}{k}\) elements of \(c_1,\ldots,c_{\binom{n}{k}}\)
occur in \(d_n(\underline{c})\); they occur once and they occur
{\em in order}. To determine which entries do not occur in
\(d_n(\underline{c})\) we look at the \(n^{\rm{th}}\) column of
the table we drew for \((n,k)\). If a \(^*\) occurs in the
\(l^{\rm{th}}\) row then (by Theorem~\ref{bigGammathm}(d)) \(c_l\)
does not occur in \(d_n(\underline{c}).\)

We finally explain how to calculate the face operators other than
\(d_0\) and \(d_n.\) For this and the next paragraph we fix
\(i\in\{1,\ldots,n-1\}\) and \(k\in\{0,\ldots,n\},\) let
\(\underline{c}\in\Gamma(C.)_{n,k}\) and write
\(\underline{c}=(c_1,\ldots,c_{\binom{n}{k}}).\) If \(k=n\) then
\(\Gamma(C.)_{n-1,k}\) is just the zero module, so
\(d_i(\underline{c})=0\). In general, the vector
\(d_i(\underline{c})\in\Gamma(C.)_{n-1,k}\) is an
\(\binom{n-1}{k}\)-tuple. By Theorem~\ref{bigGammathm}(c) each
entry of \(d_i(\underline{c})\) is either one of
\(c_1,\ldots,c_{\binom{n}{k}}\) or the sum of two of them; more
specifically each of \(c_1,\ldots,c_{\binom{n}{k}}\) occur at most
once in \(d_i(\underline{c}),\) either by itself or as part of a
sum, but might not occur at all.

We now proceed in three steps. In the first step we determine
those entries of $d_i(\underline{c})$ that consist of the sum of
two entries of $\underline{c}$ (but not yet the summands). To do
so we look at the column labelled $i-1$ in the table we have drawn
for $(n-1,k)$; if the $l^{\textrm{th}}$ row of that column has a
$\times$ mark in it then the $l^{\textrm{th}}$ entry of
$d_i(\underline{c})$ is the sum of two entries of $\underline{c}$
(by Theorem~\ref{bigGammathm}(c)(ii)). For the second and third
step we look at the column labelled $i$ in the table we have made
for $(n,k)$. In this column there are as many rows with no
$\times$ mark as there are entries of $d_i(\underline{c})$ (by
Proposition~\ref{collatingstuffinDelta}(a)(i)). The second step
now is to write the entries of $\underline{c}$ indexed by the
ordinals of these rows into $d_i(\underline{c})$ {\em in order}.
Still in the same column of the same table there are as many rows
that are marked with a $\times$ but not with a $^*$ as there are
entries of $d_i(\underline{c})$ that contain a sum (by
Proposition~\ref{collatingstuffinDelta}(a)(iii)). The final, third
step is to write the entries of $\underline{c}$ indexed by the
ordinals of these rows {\em in order} into those entries of
$d_i(\underline{c})$ we have identified in the first step to
contain a sum and join them by a plus sign with the entries we
have already made in the second step. This accomplishes
calculating $d_i(\underline{c})$ by Theorem~\ref{bigGammathm}(c).
Finally, it may be worth mentioning that if the $l^{\textrm{th}}$
row (still in the same column of the same table) contains both a
$\times$ mark and a $^*$ mark then $c_l$ does not occur in
$d_i(\underline{c})$.

%
\begin{example}\label{face and degeneracy operator calculation example}
Let \(C\rightarrow B\rightarrow A\) be a chain complex of length
$2$, placed in degrees $0,$ $1$ and $2,$ which has differential
\(\partial\). For $n\ge 0$ let $\Gamma_n := \Gamma(C\rightarrow B
\rightarrow A)_n$. For each \(n\in\{1,2,3,4,5\}\) we calculate all
the degeneracy operator \(s_i:\Gamma_{n-1} \rightarrow
\Gamma_{n}\) and all the face operators \(d_i:\Gamma_{n}
\rightarrow \Gamma_{n-1}\). But first we write write down the
tables as introduced above.


\begin{center}

Table for $(n,k)=(1,1)$:

\begin{tabular}{@{}r@{ }l|c|c|}
  &          &0&1\\
\hline
 1&$(1,1)$&\(\times^*\)&\(\times^*\)\\
\end{tabular}

\end{center}


\begin{center}

{Tables for $(n,k)=(2,1)$ and $(n,k)=(2,2)$}:

\begin{tabular}{@{}r@{ }l|l|l|l|l|l|}
  &         &0         &1         &2           \\
\hline
 1&$(1,2)$&\(\times^*\)&          &\(\times\)\\
 2&$(2,1)$&          &\(\times\)&\(\times^*\)  \\
\end{tabular}
\qquad
\begin{tabular}{@{}r@{ }l|l|l|l|l|l|}
  &          &0        &1           &2\\
  \hline
 1&$(1,1,1)$&\(\times^*\)&\(\times^*\)&\(\times^*\) \\
\end{tabular}

\end{center}

\begin{center}

{Tables for $(n,k)=(3,1)$ and $(n,k)=(3,2)$}:

\begin{tabular}{@{}r@{ }l|l|l|l|l|l|}
  &         &0         &1         &2       &3    \\
  \hline
 1&$(1,3)$&\(\times^*\)&          &          &\(\times\)\\
 2&$(2,2)$&          &\(\times\)&          &\(\times\)  \\
 3&$(3,1)$&          &          &\(\times\)&\(\times^*\)  \\
\end{tabular}
\qquad
\begin{tabular}{@{}r@{ }l|l|l|l|l|l|}
  &          &0        &1           &2           &3\\
  \hline
 1&$(1,1,2)$&\(\times^*\)&\(\times^*\)&            &\(\times\)          \\
 2&$(1,2,1)$&\(\times^*\)&            &\(\times\)  &\(\times^*\)        \\
 3&$(2,1,1)$&          &\(\times\)  &\(\times^*\)&\(\times^*\)        \\
\end{tabular}

\end{center}

\begin{center}

{Tables for $(n,k)=(4,1)$ and $(n,k)=(4,2)$}:

\begin{tabular}{@{}r@{ }l|l|l|l|l|l|}
  &         &0         &1         &2       &3         &4\\
  \hline
 1&$(1,4)$&\(\times^*\)&          &          &          &\(\times\)\\
 2&$(2,3)$&          &\(\times\)&          &          &\(\times\)\\
 3&$(3,2)$&          &          &\(\times\)&          &\(\times\)  \\
 4&$(4,1)$&          &          &          &\(\times\)&\(\times^*\)  \\
\end{tabular}
\qquad
\begin{tabular}{@{}r@{ }l|l|l|l|l|l|}
  &          &0        &1           &2           &3           &4\\
  \hline
 1&$(1,1,3)$&\(\times^*\)&\(\times^*\)&            &            &\(\times\)  \\
 2&$(1,2,2)$&\(\times^*\)&            &\(\times\)  &            &\(\times\)  \\
 3&$(1,3,1)$&\(\times^*\)&            &            &\(\times\)  &\(\times^*\)\\
 4&$(2,1,2)$&          &\(\times\)&\(\times^*\)&            &\(\times\)  \\
 5&$(2,2,1)$&          &\(\times\)  &            &\(\times\)  &\(\times^*\)\\
 6&$(3,1,1)$&          &            &\(\times\)  &\(\times^*\)&\(\times^*\)\\
\end{tabular}

\end{center}

\begin{center}

{Tables for $(n,k)=(5,1)$ and $(n,k)=(5,2)$:}

\begin{tabular}{@{}r@{ }l|l|l|l|l|l|l|}
  &       &0         &1         &2         &3         &4         &5\\
  \hline
 1&$(1,5)$&\(\times^*\)&          &          &          &          &\(\times\)\\
 2&$(2,4)$&          &\(\times\)&          &          &          &\(\times\)\\
 3&$(3,3)$&          &          &\(\times\)&          &          &\(\times\)\\
 4&$(4,2)$&          &          &          &\(\times\)&          &\(\times\)\\
 5&$(5,1)$&          &          &          &          &\(\times\)&\(\times^*\)\\
\end{tabular}
\qquad
\begin{tabular}{@{}r@{ }l|l|l|l|l|l|l|}
  &       &0         &1           &2           &3           &4         &5   \\
  \hline
 1&(1,1,4)&\(\times^*\)&\(\times^*\)&            &            &            &\(\times\)  \\
 2&(1,2,3)&\(\times^*\)&            &\(\times\)  &            &            &\(\times\)  \\
 3&(1,3,2)&\(\times^*\)&            &            &\(\times\)  &            &\(\times\)  \\
 4&(1,4,1)&\(\times^*\)&            &            &            &\(\times\)  &\(\times^*\)\\
 5&(2,1,3)&          &\(\times\)  &\(\times^*\)&            &            &\(\times\)  \\
 6&(2,2,2)&          &\(\times\)  &            &\(\times\)  &            &\(\times\)  \\
 7&(2,3,1)&          &\(\times\)  &            &            &\(\times\)  &\(\times^*\)\\
 8&(3,1,2)&          &            &\(\times\)  &\(\times^*\)&            &\(\times\)  \\
 9&(3,2,1)&          &            &\(\times\)  &            &\(\times\)  &\(\times^*\)\\
10&(4,1,1)&          &            &            &\(\times\)  &\(\times^*\)&\(\times^*\)\\
\end{tabular}
\end{center}

The face and degeneracy operators between $\Gamma_0 = A$ and
$\Gamma_1 = B \oplus A$ act as follows.
\[d_i((b;\,\,a)) = \left\{ \begin{array}{ll}\partial(b)+a & \textrm{for }i=0\\
a& \textrm{for } i=1\end{array}\right.\]
\[s_0(a) = (0;\,\,a)\]

The face and degeneracy operators between $\Gamma_1=B\oplus A$ and
$\Gamma_2 = C\oplus B^2\oplus A$ act as follows.
\[ d_i(c;\,\,b_1,b_2;\,\,a)) = \left\{ \begin{array}{lll}
(\partial(c)+b_2;\,\, \partial (b_1) +a) & \textrm{for } i=0\\
(b_1+b_2;\,\,a) & \textrm{for } i=1\\
(b_1;\,\,a)& \textrm{for } i=2
\end{array}\right.\]
\[s_i((b;\,\,a))= \left\{\begin{array}{ll}
(0,b;\,\,a) & \textrm{for } i=0\\
(b,0;\,\,a) & \textrm{for } i=1
\end{array} \right. \]

The face and degeneracy operators between $\Gamma_2=C\oplus
B^2\oplus A$ and $\Gamma_3=C^3\oplus B^3\oplus A$ act as follows.
\[d_i((c_1,c_2,c_3;\,\,b_1,b_2,b_3;\,\,a))= \left\{\begin{array}{ll}
(c_3,\partial(c_1)+b_2,\partial(c_2)+b_3;\,\,\partial(b_1)+a) &
\textrm{for } i=0\\
(c_2+c_3;\,\,b_1+b_2,b_3;\,\,a) & \textrm{for } i=1\\
(c_1+c_2;\,\,b_1,b_2+b_3;\,\,a) & \textrm{for } i=2\\
(c_1;\,\,b_1,b_2;\,\,a) & \textrm{for } i=3
\end{array}\right.\]
\[s_i((c;\,\,b_1,b_2;\,\,a)) = \left\{\begin{array}{ll}
(0,0,c;\,\,0,b_1,b_2;\,\,a) & \textrm{for } i=0\\
(0,c,0;\,\,b_1,0,b_2;\,\,a) & \textrm{for } i=1\\
(c,0,0;\,\,b_1,b_2,0;\,\,a) & \textrm{for } i=2
\end{array}\right.  \]

The face and degeneracy operators between $\Gamma_3=C^3\oplus
B^3\oplus A$ and $\Gamma_4 =C^6\oplus B^4\oplus A$ act as follows.
\begin{eqnarray*}\lefteqn{d_i((c_1,c_2,c_3,c_4,c_5,c_6;\,\,b_1,b_2,b_3,b_4;\,\,a))
}\\ &=&\left\{\begin{array}{ll}
(c_4,c_5,c_6;\,\,\partial(c_1)+b_2,\partial(c_2)+b_3,\partial(c_3)+b_4;\,\,\partial(b_1)+a)
& \textrm{for } i=0\\
(c_2+c_4,c_3+c_5,c_6;\,\,b_1+b_2,b_3,b_4;\,\,a) & \textrm{for } i=1\\
(c_1+c_2,c_3,c_5+c_6;\,\,b_1,b_2+b_3,b_4;\,\,a) & \textrm{for } i=2\\
(c_1,c_2+c_3,c_4+c_5;\,\,b_1,b_2,b_3+b_4;\,\,a) & \textrm{for } i=3\\
(c_1,c_2,c_4;\,\,b_1,b_2,b_3;\,\,a) & \textrm{for } i=4
\end{array}  \right.  \end{eqnarray*}
\[s_i((c_1,c_2,c_3;\,\,b_1,b_2,b_3;\,\,a)) = \left\{\begin{array}{ll}
(0,0,0,c_1,c_2,c_3;\,\,0,b_1,b_2,b_3;\,\,a) & \textrm{for } i=0\\
(0,c_1,c_2,0,0,c_3;\,\,b_1,0,b_2,b_3;\,\,a) & \textrm{for } i=1\\
(c_1,0,c_2,0,c_3,0;\,\,b_1,b_2,0,b_3;\,\,a) & \textrm{for } i=2\\
(c_1,c_2,0,c_3,0,0;\,\,b_1,b_2,b_3,0;\,\,a) & \textrm{for } i=3
\end{array} \right. \]

The face and degeneracy operators between $\Gamma_4 =C^6\oplus
B^4\oplus A$ and $\Gamma_5=C^{10}\oplus B^5\oplus A$ act as follows.
\begin{eqnarray*}\lefteqn{d_i((c_1,c_2,c_3,c_4,c_5,c_6,c_7,c_8,c_9,c_{10};\,\,b_1,b_2,b_3,b_4,b_5;\,\,a))}\\
&=&\left\{\begin{array}{ll}
(c_5,c_6,c_7,c_8,c_9,c_{10};\,\,\partial(c_1)+b_2,\partial(c_2)+b_3,\partial(c_3)+b_4,\partial(c_4)+b_5;\,\,\partial(b_1)+a)
\\
(c_2+c_5,c_3+c_6,c_4+c_7,c_8,c_9,c_{10};\,\,b_1+b_2,b_3,b_4,b_5;\,\,a) \\
(c_1+c_2,c_3,c_4,c_6+c_8,c_7+c_9,c_{10};\,\,b_1,b_2+b_3,b_4,b_5;\,\,a) \\
(c_1,c_2+c_3,c_4,c_5+c_6,c_7,c_9+c_{10};\,\,b_1,b_2,b_3+b_4,b_5;\,\,a) \\
(c_1,c_2,c_3+c_4,c_5,c_6+c_7,c_8+c_9;\,\,b_1,b_2,b_3,b_4+b_5;\,\,a)  \\
(c_1,c_2,c_3,c_5,c_6,c_8;\,\,b_1,b_2,b_3,b_4;\,\,a)
\end{array} \right.
\end{eqnarray*}
\begin{eqnarray*} \lefteqn{s_i((c_1,c_2,c_3,c_4,c_5,c_6;\,\,b_1,b_2,b_3,b_4;\,\,a))}\\
&=&\left\{\begin{array}{ll}
(0,0,0,0,c_1,c_2,c_3,c_4,c_5,c_6;\,\,0,b_1,b_2,b_3,b_4;\,\,a) &
\textrm{for
} i=0\\
(0,c_1,c_2,c_3,0,0,0,c_4,c_5,c_6;\,\,b_1,0,b_2,b_3,b_4;\,\,a) &
\textrm{for
} i=1\\
(c_1,0,c_2,c_3,0,c_4,c_5,0,0,c_6;\,\,b_1,b_2,0,b_3,b_4;\,\,a) &
\textrm{for
} i=2\\
(c_1,c_2,0,c_3,c_4,0,c_5,0,c_6,0;\,\,b_1,b_2,b_3,0,b_4;\,\,a) &
\textrm{for
} i=3\\
(c_1,c_2,c_3,0,c_4,c_5,0,c_6,0,0;\,\,b_1,b_2,b_3,b_4,0;\,\,a) &
\textrm{for } i=4
\end{array}  \right. \end{eqnarray*}

\end{example}

\section{Cross-effect Functors}\label{crosstheory}

In this section we summarize some definitions and results about
cross-effect functors  that are relevant to our work; see
\cite{EM} for proofs and more details.

Recall a functor \(G:\mathcal{A}\rightarrow\mathcal{B}\) between
abelian categories is called linear if, for any sequence $A_1,
\ldots, A_n$ of objects in \(\mathcal{A}\), we have the relation
\(G(\oplus_{i=1}^nA_i)=\oplus_{i=1}^nG(A_i)\) in~$\mathcal{B}$.
The main result of the theory of cross-effect functors
(Theorem~\ref{crossdecomposition}) gives us an analogous
decomposition for any {\em nonlinear} functor
\(F:\mathcal{A}\rightarrow\mathcal{B}\) with the property that
\(F(0_\mathcal{A})=0_{\mathcal{B}}\). This decomposition we get in
\(\mathcal{B}\) has a term for each subsum of the original sum in
\(\mathcal{A}\) (rather than for each summand as with a linear
functor). The terms of this sum in \(\mathcal{B}\) are given by
cross-effect functors of \(F\).


For the rest of this section we let \(F:\mathcal{A}\rightarrow\mathcal{B}\)
be a functor between an additive category \(\mathcal{A}\) and an
abelian category \(\mathcal{B}\) with \(F(0_\mathcal{A})=0_\mathcal{B}\).
The condition \(F(0_\mathcal{A})=0_\mathcal{B}\) is equivalent to the condition that
the image of any zero homomorphism in \(\mathcal{A}\) under \(F\) is a zero homomorphism in \(\mathcal{B}\).

\begin{defn}
For \(f_1,\ldots,f_n\in\Hom(A,B)\) we define the morphism
\(F(f_1\dev\ldots\dev f_n)\in \Hom(F(A),F(B))\) by the following
equation:
\[F(f_1\dev\ldots\dev f_n)=\sum_{k=1}^n\sum_{j_1<\ldots<j_k}(-1)^{n-k}F(f_{j_1}+\ldots+f_{j_k})\text{.}\]
\end{defn}

The function \(F(-\dev\ldots\dev -)\) has the following properties.
For each permutation \(\pi\) of \(\{1,\ldots,n\}\)
we have \(F(f_1\dev\ldots\dev f_n)=F(f_{\pi(1)}\dev\ldots\dev f_{\pi(n)}).\)
Whenever any of the functions \(f_i\) are zero we get \(F(f_1\dev\ldots\dev
f_n)=0\).
The function \(F(-\dev\ldots\dev -)\) is linear in each argument.
By rearranging the definition we get the relation
\(F(f_1+\ldots+f_n)=\sum_{k=1}^n\sum_{j_1<\ldots<j_k}F(f_{j_1}\dev\ldots\dev
f_{j_k})\).

\begin{note} Let \(A=A_1\oplus\ldots \oplus A_n\) be a direct sum in the
additive category \(\mathcal{A}.\) For each non-empty subset
\(\alpha=\{j_1<\ldots<j_k\}\) of \(\{1,\ldots,n\}\) and each
\(j\in\alpha\) we write \(A^\alpha\) for
\(\bigoplus_{l\in\alpha}A_l\), \(i^\alpha\) for the canonical
injection \(A^\alpha\rightarrow A\), \(p^\alpha\) for the
canonical projection \(A\rightarrow A^\alpha\), \(\psi^\alpha_j\)
for the map \(A^\alpha\rightarrow A^\alpha\),
\((a_{j_1},\ldots,a_{j_k})\mapsto (0,\ldots,0,a_j,0\ldots,0)\) and
just \(\psi_j\) if \(\alpha=\{1,\ldots,n\}.\) We also write
\((A_j,j\in\alpha)\) for the tuple \((A_{j_1},\ldots,A_{j_k})\).
\end{note}

\begin{defn}\label{cross}
The \(n^{th}\) cross-effect of \(F\) is a functor \(\mathcal{A}^n\rightarrow\mathcal{B}\).
It acts on objects by
\[
\cross_n(F)(A_1,\ldots,A_n)=F(\psi_1\dev\ldots\dev\psi_n)F(A_1\oplus\ldots \oplus A_n)\text{.}
\]
For the collection of morphisms \(f_l:A_l\rightarrow B_l\), \(
1\le l\le n\), the morphism
\[
\cross_n(F)(f_1,\ldots,f_n):
\cross_n(F)(A_1,\ldots,A_n)\rightarrow\cross_n(F)(B_1,\ldots,B_n)
\]
is induced by  $ F(f_1\oplus\ldots\oplus f_n): F(A_1 \oplus \ldots
\oplus A_n)\rightarrow F(B_1\oplus\ldots \oplus B_n)$.
\end{defn}

Definition~\ref{cross} is a technical definition of cross-effect
functors that does not really give much intuition about how one
should think of them. It is better to think of cross-effect
functors as the terms of a direct-sum decomposition as given in
Theorem~\ref{crossdecomposition} below; Theorem~\ref{crosschar}
gives us the justification of this mental picture. In a sense
Theorem~\ref{crosschar} is a converse of
Theorem~\ref{crossdecomposition}, because it says that if we have
an appropriate collection of functors which give a decomposition
of \(G(\bigoplus_{i=1}^n A)\) then they are (up to isomorphism)
the cross-effect functors of $G$.

\begin{thm}\label{crossdecomposition}

Let $A_1, \ldots, A_n \in \mathcal{A}$. The maps
\[\cross_{|\alpha|}(F)(A_j, j\in \alpha) \subseteq F(\oplus_{j\in
\alpha} A_j) \,\, \stackrel{F(i^\alpha)}{\longrightarrow}\,\,
F(A_1 \oplus \ldots \oplus A_n), \quad \alpha \subseteq \{1,
\ldots,n\},\] induce the following direct-sum decomposition of
$F(A_1 \oplus \ldots \oplus A_n)$:
\[\bigoplus_{\alpha \subseteq \{1, \ldots, n\}}
\cross_{|\alpha|}(F)(A_j, j\in \alpha) \cong F(A_1 \oplus \ldots
\oplus A_n);\] here, for each subset $\alpha = \{j_1 < \ldots
<j_{|\alpha|}\}$ of $\{1, \ldots, n\}$, the direct summand of the
left-hand side indexed by $\alpha$ corresponds to the sub-object
$F(\psi_{j_1} \dev \ldots \dev \psi_{j_{|\alpha|}}) F(A_1 \oplus
\ldots \oplus A_n)$ of the right-hand side.
\end{thm}

Cross-effect functors also have the following properties.
Whenever any of the objects \(A_j\) for \(j\in\{1,\ldots,n\}\) is the zero object
then the cross-effect module \(\cross_n(F)(A_1, \ldots, A_n)\) is also the zero object.
For each permutation \(\pi\) of \(\{1,\ldots,n\}\)
we get a natural isomorphism
\(
\cross_n(F)(A_1, \ldots, A_n)\cong\cross_n(F)(A_{\pi(1)}, \ldots, A_{\pi(n)})
.\)

\begin{defn}
If \(\cross_n(F)\) is the zero functor then we say that
\emph{\(F\) is a functor of degree less than \(n\)}. In this case
$F$  is also  of degree less than \(m\) for any $m>n$. Because of
this $F$ has a well-defined \emph{degree}.  The degree of $F$ is
either a non-negative integer or infinity.
\end{defn}

The following theorem gives us a characterization of the
cross-effect functors of \(F\) by their appearance in a direct-sum
decomposition as in Theorem~\ref{crossdecomposition}.

\begin{thm}\label{crosschar}
For each subset \(\alpha\) of $\{1, \ldots, n\}$ let \(E_\alpha\)
be a covariant functor between \(\mathcal{A}^{|\alpha|}\) and
\(\mathcal{B}\), which is zero when any of its arguments is zero.
If we have a natural isomorphism
\[
h:\bigoplus_{\alpha\subset\{1,\ldots,n\}}E_\alpha(A_j,j\in\alpha)\cong F(A_1\oplus\ldots \oplus A_n)
\]
then \(h\) maps each \(E_\alpha(A_j,j\in\alpha)\) isomorphically
to
\(F(\psi_{j_1}\dev\ldots\dev\psi_{j_{|\alpha|}})F(A_1\oplus\ldots\oplus
A_n)\). In particular we get a natural isomorphism $E_\alpha \cong
\cross_{|\alpha|}(F)$.

\end{thm}

\section{Expressing Dold-Puppe complexes in terms of cross-effect modules}\label{honourabilitynstuff}

Let \(\mathcal{A}\) be an abelian category. Previously we have
worked with the functor
\(\Gamma:\Ch_{\ge0}\mathcal{A}\rightarrow\mathcal{SA}\), now we
introduce its inverse
\(N:\mathcal{SA}\rightarrow\Ch_{\ge0}\mathcal{A}\). Let \(X.\) be
a simplicial object in \(\mathcal{A}.\) The normalized chain
complex \(N(X.)\) of $X.$ is given by
\[N(X.)_n:=X_n \Bigg/ \sum_{i=0}^{n-1}\Image s_i\text{,}\]
with its differential induced by the alternating sum of the face maps of \(X.:\)
\[
\partial=\sum^n_{i=0}(-1)^id_i :X_n\rightarrow X_{n-1}
\]
(for \(n\ge0\)). An important application of the Dold-Kan
correspondence is the construction of Dold-Puppe complexes, i.e.\
complexes of the form \(NF\Gamma(C.)\)  where \(C.\) is a chain
complex and \(F:\mathcal{A}\rightarrow\mathcal{B}\) is a functor
between abelian categories (that has been extended to the category
\(\mathcal{SA}\) in the obvious way).

In \cite{Ko1} the first-named author uses cross-effect functors to
give a description of the Dold-Puppe complex of a chain complex
\(C.={(P\rightarrow Q)}\) of length one (i.e.\ \(C_n=0\) when
\(n>1\)) in the category \(\Ch_{\ge 0}(\mathcal{A})\). Lemma~2.2
of \cite{Ko1} proves that
\[
NF\Gamma(P\rightarrow Q)_n\cong \cross_n(F)(P,\ldots,P) \oplus \cross_{n+1}(F)(Q,P,\ldots ,P)
\]
and gives an explicit description of the differential. The aim of
this section is to generalise this result and give a similar
description of Dold-Puppe complexes in terms of cross-effect
functors when the original complex is longer.

For the rest of this section we fix a functor $F: \mathcal{A}
\rightarrow \mathcal{B}$ from an additive category $\mathcal{A}$
to an abelian category $\mathcal{B}$ with the property that
$F(0_\mathcal{A}) = 0_\mathcal{B}$, we fix a chain complex $C.$ in
$\mathcal{A}$ and we fix a positive integer $n$.

The following definition introduces another way of denoting
elements of \(\Sur([n],[k])\), which will be easier to deal with
the problems in this section.

\begin{defn}\label{triangle}
Let $\mathcal{P}_n$ denote the set of subsets of $\{0, 1, \ldots,
n-1\}$. We define a bijective map $^\triangle$ as follows:
\begin{eqnarray*}
^\triangle: \amalg_{k=0}^n \Sur([n],[k]) & \rightarrow &
\mathcal{P}_n \\
\mu \in \Sur([n],[k]) & \mapsto & \mu^\triangle := \{ \max
\mu^{-1}(0), \ldots, \max \mu^{-1}(k-1) \}
\end{eqnarray*}
where $\max$ is the function that gives the maximum element of a
set. For each $k \in \{0, \ldots, n\}$, we use the symbol
$^\triangle$ also for the induced bijection between
$\Sur([n],[k])$ and the set $\mathcal{P}_n^k$ of subsets of $\{0,
\ldots, n-1\}$ of cardinality $k$.
\end{defn}

Note that we have omitted $\max \mu^{-1}(k)$ in the list of
elements of $\mu^\triangle$ because $\max \mu^{-1}(k)$ is always
equal to $n$. For every $0 \le i \le n-2$, the partition $\mu^*$
obviously begins with a partition of $i+1$ (in the sense of
Definition~\ref{initial}) if and only if $i \in \mu^\triangle$. We
will be using this observation extensively when we refer to
results of Section \ref{fdinGamma}.

\begin{defn}
We say that a subset \(\alpha\) of the disjoint union
\(\amalg_{k=0}^{n}\Sur([n],[k])\) is \emph{honourable} if
\(\cup_{\mu\in\alpha}\, \mu^\triangle = \{0,1,\ldots,n-1\}\).
\end{defn}

\begin{note}
Let \(\alpha \subset \amalg_{k=0}^{n}\Sur([n],[k]).\) For each
\(k\in\{0,\ldots,n\}\) we write \(\alpha_k\) for the intersection
\(\alpha\cap\Sur([n],[k])\).
For \(C_0,\ldots,C_n\in\mathcal{A}\) we write
\((C_{0,\alpha_0},\ldots,C_{n,\alpha_n})\) for the following
\(|\alpha|\text{-tuple}:\)
\[(\underbrace{C_{0},\ldots,C_{0}}_{|\alpha_0| \textrm{ times}},\,\,
\ldots,\,\,\underbrace{C_n,\ldots,C_n}_{|\alpha_n| \textrm{ times}}).\]
\end{note}

\begin{prop}\label{hon not degen}
We have a canonical isomorphism
\[
NF\Gamma(C.)_n\cong
\bigoplus_{\alpha\subset\amalg_{k=0}^n\Sur([n],[k]),\,\,\text{\rm{\(\alpha\)
is
honourable}}}\cross_{|\alpha|}(F)(C_{0,\alpha_0},\ldots,C_{n,\alpha_n})\text{.}
\]
\end{prop}

\begin{proof}
Using the definitions of \(N\) and \(\Gamma\) we see that
\begin{align*}
NF\Gamma (C.)_n
=F\Big({\bigoplus ^{n} _{k=0} \bigoplus_{ \mu \in \Sur([n],[k])} C_k}\Big)\Bigg/\sum_{i=0}^{n-1}\Image F(s_i)
\text{.}
\end{align*}
Theorem~\ref{bigGammathm}(a) tells us that \(\Image s_i=\bigoplus
^{n} _{k=0} \bigoplus_{\mu \in (S^{n,k}_i)^C} C_k\) which is a
subsum of the sum \(\bigoplus ^{n} _{k=0} \bigoplus_{\mu \in
\Sur([n],[k])}C_k,\) so Theorem~\ref{crossdecomposition} tells us
that \(F(\Image s_i)\cong\Image F(s_i)\).  So we get
\begin{align*}
NF\Gamma (C.)_n
\cong F\Big({\bigoplus ^{n} _{k=0} \bigoplus_{ \mu \in \Sur([n],[k])} C_k}\Big)\Bigg/\sum_{i=0}^{n-1}F(\Image s_i)
\text{.}
\end{align*}
Expanding the numerator in terms of cross effects according to
Theorem~\ref{crossdecomposition} we get the formula
\[
F\Big({\bigoplus ^{n} _{k=0} \bigoplus_{ \mu \in \Sur([n],[k])}
C_k}\Big)= \bigoplus_{\alpha\subseteq\amalg_{k=0}^n\Sur([n],[k])}
\cross_{|\alpha|} (F)(C_{0,\alpha_0},\ldots,C_{n,\alpha_n})
\text{.}
\]
Now using Theorem~\ref{bigGammathm}(a) to give us an expression
for \(\Image(s_i)\) we expand the denominator in terms of cross
effects and we see that:
\[
F(\Image s_i) = F\Big(\bigoplus_{k=0}^n \bigoplus_{\mu\in\Sur
([n],[k])\setminus S^{n,k}_i} C_{k,\mu}\Big) =\bigoplus_{\alpha}
\cross_{|\alpha|} (F)(C_{0,\alpha_0},\ldots,C_{n,\alpha_n})
\text{,}
\]
where the last sum ranges over all subsets \(\alpha\subset
\amalg_{k=0}^n\Sur([n],[k])\) where \( i \not\in
\cup_{\mu\in\alpha}\, \mu^\triangle.\) From this we see that
\(\cross_{|\alpha|}(F)(C_{0,\alpha_0},\ldots,C_{n,\alpha_n})\) is
not a direct summand of \(\Image F(s_i)\) if and only if \( i \in
\cup_{\mu\in\alpha}\, \mu^\triangle.\) A module is a direct
summand of \(NF\Gamma(C.)_n\) if and only if it is not a direct
summand of \(\sum_{i=0}^{n-1} \Image F(s_i),\) and hence we see
the desired result.
\end{proof}

Although the expression for $NF\Gamma(C.)_n$ given in the previous
proposition is quite compact it still contains many vanishing
terms: whenever $|\alpha|$ is bigger than than the degree of $F$
or $\alpha_k$ is non-empty for $k$ bigger than the length of $C.$,
the term $\cross_{|\alpha|}(F)(C_{0,\alpha_0}, \ldots C_{n,
\alpha_n})$ vanishes. The rest of this section is devoted to the
problem of quickly finding those honourable subsets $\alpha$ for
which $\cross_{|\alpha|}(F)(C_{0,\alpha_0}, \ldots C_{n,
\alpha_n})$ does not vanish. A first (still rather rough) result
in this direction is Corollary~\ref{length of DP complex} below.
Later we will describe an algorithm that produces the relevant
honourable subsets fairly quickly.

\begin{prop}\label{honour inequality}
\begin{enumerate}[(a)]
\item Let \(\alpha\) be an honourable subset of
\(\amalg_{k=0}^{n}\Sur([n],[k])\). Then we have the inequality
\(\sum_{k=0}^nk|\alpha_k|\ge n\). \item Conversely let
\((a_0,\ldots,a_n)\in {\mathbb{N}_0}^{n+1}\) with
\(a_k\le\binom{n}{k}\) for each \(k\in\{0,\ldots,n\}\). If
\(\sum_{k=0}^nk a_k\ge n\) then there is some honourable subset
\(\alpha\) of \(\amalg_{k=0}^{n}\Sur([n],[k])\) with
\(|\alpha_k|=a_k\) for each \(k\in \{0,\ldots,n\}.\)
\end{enumerate}
\end{prop}

\begin{proof}
Firstly we prove part~(a).  We know \(\alpha\) is honourable, so
by definition
\[
\cup_{k=0}^n\cup_{\mu\in\alpha_k} \, \mu^\triangle =
\{0,1,\ldots,n-1\}\text{.}
\]
Hence
\[\sum_{k=0}^n k|\alpha_k| =
\sum_{k=0}^n\sum_{\mu\in\alpha_k} |\mu^\triangle| \ge
|\{0,1,\ldots,n-1\}|=n.
\]

Now we prove part~(b). Because
\(|\{0,\ldots,n-1\}|=n\le\sum_{k=0}^nka_k\) and $a_k \le
\binom{n}{k}$ we can cover the set \(\{0,\ldots,n-1\}\) using
\(a_1\) subsets of cardinality~1, \(a_2\) subsets of
cardinality~2, \ldots, \(a_{n-1}\) subsets of cardinality~\(n-1\)
and \(a_n\) subsets of cardinality~\(n\). Take such a covering
\(\beta\) and define $\alpha$ to be the preimage of $\beta$ under
the map $^\triangle: \amalg_{k=0}^n \Sur([n],[k]) \rightarrow
\mathcal{P}_n$ introduced in Definition~\ref{triangle}. Then
$\alpha$ has the desired properties.
\end{proof}

\begin{cor}\label{length of DP complex}
The length of the Dold-Puppe complex \(NF\Gamma(C.)\) is less than
or equal to the product $ld$ of the length $l$ of \(C.\) and the
degree $d$ of \(F\). Equality is achieved if the module
\(\cross_{d}(F)(C_l,\ldots,C_l)\) is not the zero module.
\end{cor}

\begin{proof}
Proposition~\ref{hon not degen} tells us that
\[
NF\Gamma(C.)_n\cong
\bigoplus_{\alpha\subset\amalg_{k=0}^n\Sur([n],[k]),\text{\rm{\(\alpha\) is honourable}}}\cross_{|\alpha|}(F)(C_{0,\alpha_0},\ldots,C_{n,\alpha_n})\text{.}
\]
If \(|\alpha|>d\) then \(\cross_{|\alpha|}(F)(C_{0,\alpha_0},\ldots,C_{n,\alpha_n})\) vanishes.
Also the properties of cross-effects tell us if any of the modules are zero then cross-effect modules involving them will also vanish,
in particular any which involve any copies of \(C_{l'}\) where \(l'>l\) vanish.
So the only non-zero cross-effect modules in \(NF\Gamma(C.)_n\) are those
which correspond to subsets of \(\amalg_{k=0}^{\min\{n,l\}}\Sur([n],[k])\)
that are honourable and of cardinality \(d\) or less.

It therefore suffices to show that, if $n > ld$, there does not
exist any honourable subset $\alpha$ of
$\sum_{k=0}^{\min\{n,l\}}\Sur([n],[k])$ that satisfies $|\alpha|
\le d$. Suppose $\alpha$ is such a subset. As $|\alpha_k| = 0$ for
$k > \min\{n,l\} = l$ (we may assume $d \ge 1$) we obtain
\[\sum_{k=0}^n |\alpha_k| k = \sum_{k=0}^l |\alpha_k|k \le
\sum_{k=0}^l |\alpha_k|l = l |\alpha| \le ld < n. \] This
contradicts Proposition~\ref{honour inequality}(a).

To prove equality is achieved if \(\cross_d(F)(C_l,\ldots,C_l)\)
is not the zero module, we set \(n=dl,\) \(a_l=d\) and \(a_k=0\)
if \(k\ne l\). Proposition~\ref{honour inequality}(b) tells us
that there is some honourable set \(\alpha\subset
\amalg_{k=0}^{n}\Sur([n],[k])\) with \(|\alpha_k|=a_k\) for each
\(k\in\{0,\ldots,n\}\). This condition tells us that
\(\alpha\subset \Sur([n],[l])\). So
\(\cross_{|\alpha|}(F)(C_{0,\alpha_0},\ldots,C_{n,\alpha_n})=\cross_{d}(F)(C_{l,\alpha_l})\).
This is non-zero by assumption and a direct summand of
\(NF\Gamma(C.)_n\) because of our choice of \(\alpha\).
\end{proof}

The following definition will be useful in describing the
algorithm mentioned above.

\begin{defn}\label{superfluous}
(a) We define a total order on the powerset $\mathcal{P}_n$ of
$\{0, 1, \ldots, n-1\}$ as follows. Let $x=\{i_1 < \ldots < i_k\}$
and $y = \{j_1 < \ldots < j_{k'}\}$ be sets in $\mathcal{P}_n$.
Then $x\le y$ if and only if $k' < k$ or ($k'=k$ and $(i_1,
\ldots, i_k)\le (j_1, \ldots, j_k)$ in the lexicographic
ordering). \\
(b) Let $T$ be a subset of $\mathcal{P}_n$ and let $x$ be a set in
$T$. We say that $x$ is {\em superfluous} in $T$ if $\cup_{y \in
T}\,y = \cup_{y\in
T\setminus \{x\}}\,y$.\\
(c) We say that an honourable subset $\alpha$ of $\amalg_{k=0}^n
\Sur([n],[k])$ is {\em minimal} if $\alpha^\triangle$ does not
contain any superfluous sets.
\end{defn}

Recall that we have introduced a total order on $\Sur([n],[k])$ in
Definition~\ref{order} for each $k \in \{0,1, \ldots, n\}$. It is
easy to see that the bijection $^\triangle:\Sur([n],[k])
\rightarrow \mathcal{P}_n^k$ is order preserving. The following
easy procedure is an efficient way for checking whether a subset
$T$ of $\mathcal{P}_n$ contains superfluous sets, particularly in
the context of the algorithm described later.

\begin{proc}\label{procedure}
Let $T$ be a subset of $\mathcal{P}_n$. We first order the sets in
$T$ using the ordering introduced in
Definition~\ref{superfluous}(a), say $T = \{x_1 < \ldots < x_m\}$.
For each $r=2, \ldots, m$ and for each $i \in x_r$ we then check
whether $i \in x_1 \cup \ldots \cup x_{r-1}$. If so, we underline
$i$ in each of the sets $x_1, \ldots, x_r$ where it occurs. There
are two ways for this procedure to stop: (1) we perform the check
(and if necessary the underlining) described above for each
\(r\in\{2,\ldots,m\}\) and each \(i\in x_r\) and at each stage we
find that no set in $T$ has all of its elements underlined; (2) at
some point we find some set $x$ in $T$ with each of its elements
underlined. In case~(1) no superfluous sets are contained in $T$;
in case~(2) the set $x$ is superfluous in $T$.

\end{proc}

\begin{example} Let $n=4$.\\
(a) Applying Procedure~\ref{procedure} to $T= \{\{0\}, \{0,3\},
\{0,1\}\}$ we first obtain $\{\underline{0},1\} <
\{\underline{0},3\}$ and then $\{\underline{0},1\} <
\{\underline{0}, 3\} < \{\underline{0}\}$. Hence the last set
$\{\underline{0}\}$ is superfluous. \\
(b) Applying Procedure~\ref{procedure} to $T = \{\{0,1\}, \{1,2\},
\{2,3\}\}$ we first obtain $\{0, \underline{1}\} <
\{\underline{1},2\}$ and then $\{0, \underline{1}\} <
\{\underline{1}, \underline{2}\} < \{\underline{2},3\}$. Hence the
second set $\{1,2\}$ is superfluous. \\
(c) Applying Procedure~\ref{procedure} to $T=\{\{0,1,2\},
\{1,3\}\}$ we obtain $\{0, \underline{1}, 2\} < \{1, 3\}$. Hence
none of the
sets in $T$ is superfluous. \\
(d) Procedure~\ref{procedure} applied to $T=\{\{0,1\}, \{1,2\},
\{1\}, \{2\}, \{3\}\}$ stops at $\{0, \underline{1}\} <
\{\underline{1},2\} < \{\underline{1}\}$.
\end{example}

We now describe an algorithm which finds all {\em minimal}
honourable subsets of the set $\amalg_{k=0}^n \Sur([n],[k])$ in an
efficient way. Via the bijection $^\triangle: \amalg_{k=0}^n
\Sur([n],[k]) \rightarrow \mathcal{P}_n$ (see
Definition~\ref{triangle}) this amounts to finding all subsets $T$
of $\mathcal{P}_n$ such that $\cup_{x \in T}\, x = \{0, 1, \ldots,
n-1\}$ and such that $T$ does not contain any superfluous sets. We
below first inductively define a finite list $T_1, T_2, \ldots$ of
subsets of $\mathcal{P}_n$. From the construction it will be
immediately clear that $T_1, T_2, \ldots$ is the list of all
subsets of $\mathcal{P}_n$ which do not contain any superfluous
sets. We finally just discard those subsets from the list which
are not honourable.

\begin{defn}\label{list}
We inductively define a finite list $T_1, T_2, \ldots $ of subsets
of $\mathcal{P}_n$ containing no superfluous sets as follows. Let
$T_1 := \{\{0, 1, \ldots, n-1\}\}$ and suppose $T_1, \ldots, T_m$
have already been defined. We write $T_m$ in the form $\{x_1 <
\ldots <x_r\}$ with some sets $x_1, \ldots, x_r$ in
$\mathcal{P}_n$. If $r=1$ and $x_1 =\{n-1\}$, i.e.\ if $x_1$ is
the maximal set in $\mathcal{P}_n\setminus \{\emptyset\}$, the
list $T_1, \ldots, T_m$ is complete. We now assume this is not the
case. If $T_m$ is not honourable then (since by construction $T_m$
contains no superfluous set) there exists a set $y$ in
$\mathcal{P}_n$ bigger than $x_r$ such that $\{x_1 < \ldots < x_r
< y\}$ does not contain any superfluous set; we choose $y$ to be
minimal with this property and define $T_{m+1}:= \{x_1 < \ldots <
x_r < y\}$. If $T_m$ is honourable there exists an index $s \in
\{1, \ldots, r\}$ and a set $y$ in $\mathcal{P}_n$ bigger than
$x_{s-1}$ such that $\{x_1 < \ldots < x_{s-1} < y\}$ does not
contain any superfluous set. We choose $s \in \{1, \ldots, r\}$ to
be maximal and $y \in \mathcal{P}_n$ to be minimal with this
property and define $T_{m+1}:=\{x_1 < \ldots < x_{s-1} < y\}$.
\end{defn}

\begin{example}\label{list3}
For $n=3$ the previous definition gives the following list $T_1,
T_2, \ldots$ of subsets of $\mathcal{P}_3$. Following the
convention introduced in Procedure~\ref{procedure} we underline
certain elements to be able to easily detect superfluous sets.
\\
$T_1= \{\{0,1,2\}\}$, $T_2 = \{\{0,1\}\}$, $T_3 =
\{\{\underline{0},1\} < \{\underline{0},2\}\}$, $T_4 = \{\{0,
\underline{1}\} < \{\underline{1}, 2\}\}$, $T_5 = \{\{0, 1\} <
\{2\}\}$, $T_6 = \{\{0,2\}\}$, $T_7 = \{\{0, \underline{2}\} <
\{1, \underline{2}\}\}$, $T_8 = \{\{0,2\} < \{1\}\}$, $T_9 =
\{\{1,2\}\}$, $T_{10} = \{\{1,2\} < \{0\}\}$, $T_{11} =
\{\{0\}\}$, $T_{12} = \{\{0\} < \{1\}\}$, $T_{13} = \{\{0\} <
\{1\} < \{2\}\}$, $T_{14} = \{\{0\} < \{2\}\}$, $T_{15} =
\{\{1\}\}$, $T_{16} = \{\{1\} < \{2\}\}$, $T_{17} = \{\{2\}\}$. \\
The subsets $T_1$, $T_3$, $T_4$, $T_5$, $T_7$, $T_8$, $T_{10}$ and
$T_{13}$ correspond to minimal honourable subsets of
$\amalg_{k=0}^3 \Sur([n],[k])$.
\end{example}

As explained earlier, in order to calculate the direct-sum
decomposition in Proposition~\ref{hon not degen} there is no need
to find those honourable subsets $\alpha$ of $\amalg_{k=0}^n
\Sur([n],[k])$ for which $\alpha_k$ is non-empty for $k$ bigger
than the length $l$ of $C.$. In other words, rather than starting
the inductive procedure in Definition~\ref{list} at the smallest
set $\{0,1, \ldots, n-1\}$ in $\mathcal{P}_n$ it suffices to begin
at $\{0, 1, \ldots, \min\{n,l\}-1\}$.

\begin{example}\label{list4}
In this example we apply Definition~\ref{list} in the case $n=4$.
We begin the induction only at $\{\{0,1\}\}$ rather than at
$T_1=\{\{0,1,2,3\}\}$, i.e.\ we assume $l=2$. For simplicity, we
omit the external brackets for $T_i$, we in fact omit the name
$T_i$ as well (but  keep the order of the list of course) and we
moreover write down only subsets of $\mathcal{P}_n$ which
correspond to
minimal {\em honourable} subsets. The result is as follows.\\
$\{0,1\} < \{0,2\}< \{0,3\}$, $\{0,1\} < \{0,2\} < \{3\}$,
$\{0,1\} < \{0, 3\} < \{2\}$, $\{0, 1\} < \{1,2\} < \{1,3\}$,
$\{0,1\} < \{1,2\} < \{3\}$, $\{0,1\} < \{1,3\} < \{2\}$, $\{0,1\}
< \{2, 3\}$, $\{0,1\} <\{2\}<\{3\}$, $\{0,2\}<\{0,3\}<\{1\}$,
$\{0,2\}<\{1,2\}<\{2,3\}$, $\{0,2\}<\{1,2\}<\{3\}$,
$\{0,2\}<\{1,3\}$, $\{0,2\} < \{2,3\}<\{1\}$,
$\{0,2\}<\{1\}<\{3\}$, $\{0,3\}< \{1,2\}$, $\{0,3\} < \{1,3\} <
\{2,3\}$, $\{0,3\} < \{1,3\} < \{2\}$, $\{0,3\} < \{2,3\} <
\{1\}$, $\{0,3\} < \{1\} <\{2\}$, $\{1,2\} < \{1,3\} < \{0\}$,
$\{1,2\} < \{2,3\} < \{0\}$, $\{1,2\} < \{0\} < \{3\}$, $\{1,3\} <
\{2,3\} < \{0\}$, $\{1,3\} < \{0\} < \{2\}$, $\{2,3\} < \{0\} <
\{1\}$, $\{0\} < \{1\} < \{2\} < \{3\}$.
\end{example}

The object of the following example is to illustrate the methods
developed earlier in this paper.

\begin{example}\label{final}
Let $R$ be a commutative ring and let $C
\,\,\stackrel{\partial}{\rightarrow} \,\, B \,\,
\stackrel{\partial}{\rightarrow}\,\, A$ be a chain complex of
$R$-modules of length~2 (sitting in degrees 0, 1 and 2). The goal
of this example is to explicitly write down the Dold-Pupppe
complex $Q.:= N \Sym^2\Gamma(C \rightarrow B \rightarrow A)$. We
proceed in two steps. In the first step we write down the object
$Q_n$ for $n=0,1, \ldots$ (using the method developed in this
section) and in the second step we write down the differential
$\Delta: Q_n \rightarrow Q_{n-1}$ for $n=1,2, \ldots$ (using the
calculations made at the end of Section~2).

\noindent By Corollary~\ref{length of DP complex} the chain
complex $Q.$ is of length~4. From Proposition~\ref{hon not degen}
we immediately get $D_0 = \Sym^2(A)$. To calculate $D_n$ for
$n=1,2,3,4$ we first find all honourable subsets of
$\amalg_{k=0}^n \Sur([n],[k])$. The subsets of $\mathcal{P}_n$
listed below correspond to minimal honourable subsets of
$\amalg_{k=0}^n \Sur([n],[k])$. As explained earlier before
Example~\ref{list4} we write down only those subsets $T$ of
$\mathcal{P}_n$ whose sets contain at most 2 elements. We
furthermore write down only those subsets $T$ of $\mathcal{P}_n$
which contain at most 2 sets (because the degree of $\Sym^2$ is
2). As in Example~\ref{list4} we omit the exterior brackets. For
$n=3$ and $n=4$ we use Examples~\ref{list3} and \ref{list4},
respectively.

\noindent$n=1: \quad \{0\}$\\
$n=2: \quad \{0,1\}$, $\{0\} < \{1\}$\\
$n=3: \quad \{0,1\} < \{0,2\}$, $\{0,1\} < \{1,2\}$, $\{0,1\} <
\{2\}$,
$\{0,2\} < \{1,2\}$, \\
$\phantom{n=3: \quad ,}$ $\{0,2\}<\{1\}$, $\{1,2\}<\{0\}$\\
$n=4: \quad \{0,1\} < \{2,3\}$, $\{0,2\}<\{1,3\}$, $\{0,3\} <
\{1,2\}$

\noindent We finally add to these lists those subsets $T$ of
$\mathcal{P}_n$ which correspond to non-minimal honourable
subsets. As above we are only interested in subsets $T$ of
$\mathcal{P}_n$ of cardinality at most~2. Hence the lists for
$n=3$ and $n=4$ do not change. For $n=1$ and $n=2$ the completed
lists are as follows.

\noindent$n=1: \quad \{0\}$, $\{0\} < \emptyset$\\
$n=2: \quad \{0,1\}$, $\{0, 1\} < \{0\}$, $\{0,1\} < \{1\}$,
$\{0,1\} < \emptyset$, $\{0\}<\{1\}$

\noindent(By the way, this also illustrates that it is more
efficient to first find the minimal honourable subsets and then to
add the relevant non-minimal hounourable subsets than to
immediately go for all
honourable subsets.)\\
Hence the objects $Q_0, \ldots, Q_4$ are as follows.

\noindent$Q_0= \Sym^2(A)$\\
$Q_1= \Sym^2(B_1) \oplus B_1 \otimes A$\\
$Q_2= \Sym^2(C_1) \oplus C_1 \otimes B_1 \oplus C_1 \otimes B_2
\oplus C_1 \otimes A \oplus B_1 \otimes B_2$\\
$Q_3= C_1 \otimes C_2 \oplus C_1 \otimes C_3 \oplus C_1 \otimes
B_3 \oplus C_2 \otimes C_3 \oplus C_2 \otimes B_2 \oplus C_3
\otimes B_1$\\
$Q_4 = C_1 \otimes C_6 \oplus C_2 \otimes C_5 \oplus C_3 \otimes
C_4$

Here, for instance the module $C_5$ in $Q_4$ refers to the
$5^{\textrm{th}}$ copy of the module $C$ in $\Gamma(C \rightarrow
B \rightarrow A)_4 = C^6 \oplus B^4 \oplus A$, using the ordering
of copies of $C$ introduced in Section~2. \\
We finally turn to the differential $\Delta:Q_n \rightarrow
Q_{n-1}$ for $n=1,2,3,4$. It is induced by $\sum_{i=0}^n (-1)^i
d_i$ (see Section~3). Here, $d_i$ denotes the $i^\textrm{th}$ face
operator in $\Sym^2 \Gamma (C \rightarrow B \rightarrow A)$; i.e.\
$d_i$ is the symmetric square of the $i^\textrm{th}$ face operator
in $\Gamma(C \rightarrow B \rightarrow A)$. Using the calculation
given in Example~\ref{face and degeneracy operator calculation
example} and some elementary facts about the cross-effects of
$\Sym^2$ we obtain the following action of $d_i$ on each direct
summand of $Q_n$ for $n=1,2,3,4$.

\noindent$\begin{array}{lll}
n=1: & d_0:& \Sym^2(B_1) \rightarrow
\Sym^2 (A), \quad   bb' \mapsto
\partial(b) \partial(b')\\
&& B_1 \otimes A \rightarrow \Sym^2(A), \quad b\otimes a \mapsto
\partial (b) a\\
& d_1: & \textrm{acts as the zero map on } Q_1\\
\end{array}$

\noindent$\begin{array}{lll}
n=2:& d_0:& \Sym^2(C_1) \rightarrow
\Sym^2(B_1), \quad cc' \mapsto
\partial(c)\partial(c')\\
&& C_1 \otimes B_1 \rightarrow B_1 \otimes A, \quad c \otimes b
\mapsto \partial(c) \otimes \partial(b)\\
&& C_1 \otimes B_2 \rightarrow \Sym^2(B_1), \quad c \otimes b
\mapsto
\partial(c) b\\
&& C_1 \otimes A \rightarrow B_1\otimes A, \quad c\otimes a
\mapsto
\partial(c) \otimes a\\
&& B_1 \otimes B_2 \rightarrow B_1 \otimes A, \quad b \otimes b'
\mapsto b' \otimes \partial (b)\\
& d_1: & \textrm{acts as the zero map on the first four direct summands of } Q_2\\
&& B_1 \otimes B_2 \rightarrow \Sym^2(B_1),\quad b\otimes b'
\mapsto
bb'\\
& d_2: &\textrm{acts as the zero map on } Q_2\\
\end{array}$

\noindent$\begin{array}{lll}
n=3:& d_0:& C_1 \otimes C_2
\rightarrow B_1 \otimes B_2, \quad c
\otimes c' \mapsto \partial(c) \otimes \partial(c')\\
&& C_1 \otimes C_3 \rightarrow C_1 \otimes B_1, \quad c \otimes c'
\mapsto c' \otimes \partial(c)\\
&& C_1 \otimes B_3 \rightarrow B_1 \otimes B_2, \quad c \otimes b
\mapsto \partial(c) \otimes b\\
&& C_2 \otimes C_3 \rightarrow C_1 \otimes B_2, \quad c \otimes c'
\mapsto c' \otimes \partial(c)\\
&& C_2 \otimes B_2 \rightarrow B_1 \otimes B_2, \quad c \otimes b
\mapsto b \otimes \partial(c)\\
&& C_3 \otimes B_1 \rightarrow C_1 \otimes A, \quad c \otimes b
\mapsto c \otimes \partial(b)\\
& d_1: & \textrm{acts as the zero map on the first three direct
summands of } Q_3\\
&&C_2 \otimes C_3 \rightarrow \Sym^2(C_1), \quad c \otimes c'
\mapsto
cc'\\
&& C_2 \otimes B_2 \rightarrow C_1 \otimes B_1, \quad c \otimes b
\mapsto c \otimes b\\
&& C_3 \otimes B_1 \rightarrow C_1 \otimes B_1, \quad c \otimes b
\mapsto c \otimes b\\
&d_2:& \textrm{acts as the zero map on the second, fourth and
sixth direct summand of } Q_3\\
&& C_1 \otimes C_2 \rightarrow \Sym^2(C_1), \quad c \otimes c'
\mapsto cc'\\
&& C_1 \otimes B_3 \rightarrow C_1 \otimes B_2, \quad c \otimes b
\mapsto c \otimes b\\
&& C_2 \otimes B_2 \rightarrow C_1 \otimes B_2, \quad c \otimes b
\mapsto c\otimes b\\
& d_3: &\textrm{acts as the zero map on } Q_3\\
\end{array}$

\noindent$\begin{array}{lll}
n=4:& d_0:& C_1 \otimes C_6
\rightarrow C_3 \otimes B_1, \quad c
\otimes c' \mapsto c' \otimes \partial(c)\\
&& C_2 \otimes C_5 \rightarrow C_2 \otimes B_2, \quad c \otimes c'
\mapsto c' \otimes \partial(c)\\
&& C_3 \otimes C_4 \rightarrow C_1 \otimes B_3, \quad c \otimes c'
\mapsto c' \otimes \partial(c)\\
& d_1:& \textrm{acts as the zero map on the first direct summand
of } Q_4\\
&& C_2 \otimes C_5 \rightarrow C_1 \otimes C_2, \quad c \otimes c'
\mapsto c\otimes c'\\
&& C_3 \otimes C_4 \rightarrow C_1 \otimes C_2, \quad c \otimes c'
\mapsto c' \otimes c\\
&d_2:& \textrm{acts as the zero map on the last direct summand of
} Q_4\\
&& C_1 \otimes C_6 \rightarrow C_1 \otimes C_3, \quad c \otimes c'
\mapsto c\otimes c'\\
&& C_2 \otimes C_5 \rightarrow C_1 \otimes C_3, \quad c \otimes c'
\mapsto c \otimes c'\\
& d_3:& \textrm{acts as the zero map on the first direct summand
of } Q_4\\
&& C_2 \otimes C_5 \rightarrow C_2 \otimes C_3, \quad c \otimes c'
\mapsto c \otimes c'\\
&& C_3 \otimes C_4 \rightarrow C_2 \otimes C_3, \quad c \otimes c'
\mapsto c \otimes c'\\
& d_4: & \textrm{acts as the zero map on } Q_4
\end{array}$
\end{example}


\begin{thebibliography}{EM}
\bibitem[DP]{DP} \textsc{A.\ Dold} and \textsc{D.\ Puppe},
  \textit{Homologie nicht-additiver Funktoren. Anwendungen},
  Ann.\ Inst.\ Fourier Grenoble  11 (1961), 201-312.

\bibitem[EM]{EM} \textsc{S.\ Eilenberg} and \textsc{S.\ Mac Lane},
  \textit{On the groups $H(\Pi, n)$, II},
  Ann.\ of Math.\ 60 (1954), 49-139.


\bibitem[JM]{JM} \textsc{B.\ Johnson} and \textsc{R.\ McCarthy},
  \textit{Linearization, Dold-Puppe stabilization, and Mac Lane's Q-construction},
  Trans. Amer. Math. Soc. 350 (1998), 1555-1593.




\bibitem[K\"o]{Ko1} \textsc{B. K\"ock},
  \textit{Computing the homology of Koszul complexes},
  Trans.\ Amer.\ Math.\ Soc.~353 (2001), 3115-3147.

\bibitem[W]{Weibel} \textsc{C. A. Weibel},
  \textit{An introduction to homological algebra},
  Cambridge Stud.\ Adv.\ Math.\ 38, Cambridge University Press (1994).



\end{thebibliography}
\end{document}